\documentclass[11pt] {article}      
\usepackage{amsmath,amssymb, amsthm, eucal,accents}  
\usepackage{graphicx}
\usepackage{kbordermatrix}
\usepackage{txfonts}
\usepackage[T1]{fontenc}     
\usepackage{enumerate}       
\usepackage{float}                 
\usepackage{mathtools}
\usepackage{tikz}
\usetikzlibrary{matrix, arrows, calc, fit}
\usetikzlibrary{arrows, decorations.markings}
\tikzstyle{AArrow} = [thick, decoration={markings,mark=at position 1 with {\arrow[semithick]{open triangle 60}}},%
   double distance=1.4pt, shorten >= 5.5pt,
   preaction = {decorate},
   postaction = {draw,line width=1.4pt, white,shorten >= 4.5pt}]
\tikzstyle{AArroww} = [semithick, white,line width=1.4pt, shorten >= 4.5pt]
\usepackage[strict]{changepage}
\usepackage{relsize}
\theoremstyle{plain}
\newtheorem{theorem}{Theorem}[section]            
\newtheorem{proposition}[theorem]{Proposition}  

\theoremstyle{definition}
\newtheorem{definition}[theorem]{Definition}

\numberwithin{theorem}{section}
\numberwithin{equation}{section}
\numberwithin{figure}{section}
\newcommand{\gaction}[2]{\genfrac{}{}{0.5pt}{}{#1}{#2}%
                        \!\lower2pt\hbox{\rotatebox[origin=c]{-90}{{$\looparrowright$}}}}
\newcommand{\dgaction}{\displaystyle\gaction}
\newcommand{\dotfraction}[2]{\genfrac{}{}{0.5pt}{}{#1}{#2}%
                        \!\lower.5pt\hbox{{$\circ$}}}
\def\mathbi{\boldsymbol}   		
\def\splay{\hbox{splay}}
\def\rmspan{\hbox{span}\,}
\def\rIm{\hbox{Im\,}}
\def\q{\partial}
\def\card{\hbox{card}\,}
\def\Pu{\hbox{Pu}\,}

\def\diag{\hbox{diag}}
\def\z{\!\!\!\!\!}
\renewcommand\appendix{\par
  \setcounter{section}{0}
  \setcounter{subsection}{0}
  \setcounter{figure}{0}
  \setcounter{table}{0}
  \renewcommand\thesection{Appendix \Alph{section}}
  \renewcommand\thefigure{\Alph{section}\arabic{figure}}
  \renewcommand\thetable{\Alph{section}\arabic{table}}
}
\usepackage{titlesec}                                                           
\titlelabel{\thetitle.\ }                                                           
\titleformat*{\section}{\fontsize{14pt}{14pt} \bf}                   
\usepackage{fancyhdr}
\usepackage{sidecap}  
\usepackage[hang,small,sc]{caption}

\newcommand*{\nfrac}[2]{\genfrac{}{}{0pt}{}{#1}{#2}}
\def\smalll{\scriptsize}

\def\ontop{\accentset}
\def\QED{ $\square$}

\begin{document}

\title{\bf Krawtchouk matrices from the Feynman path integral and from the split quaternions}

\author{Jerzy Kocik                  
\\ \small Department of Mathematics
\\ \small Southern Illinois University, Carbondale, IL62901
\\ \small jkocik{@}siu.edu  }
\date{}

\maketitle

\begin{abstract}
\noindent
An interpretation of Krawtchouk matrices in terms of discrete version of the Feynman path integral is given.
Also, an algebraic characterization in terms of the algebra of split quaternions is provided.
The  resulting properties  include an easy inference of the spectral decomposition.  
It is also an occasion for an expository clarification of the role of 
Krawtchouk matrices in different areas,
including quantum information. 
\\[3pt]
\\

\end{abstract}



\qquad\qquad 1. What are Krawtchouk matrices

\qquad\qquad 2. Counting, controlled and erratic 

\qquad\qquad 3. What is quantum computing 

\qquad\qquad 4. Ehrenfest urn problem: beyond Kac's solution  

\qquad\qquad 5. Topological interpretation of Krawtchouk matrices via ``twistons''

\qquad\qquad 6. Feynman sum over paths interpretation 

\qquad\qquad 7. Quaternions and related Lie groups and algebras

\section{What are Krawtchouk matrices}

Motivated by applications in data analysis and experiment design, 
Mykhailo Krawtchouk introduced a family of orthogonal polynomials \cite{Kra1, Kra2}, 
which could be defined in terms of hypergeometric functions as
$$
k_n^{(p)}(x,N) = {}_2F_1\left(\nfrac{-n}{-x} \Bigg|-N; \frac{1}{p}\right)
$$
But such a description misses their fundamental nature and organic simplicity.
In 1986, Nirmal Bose defined matrices with entries corresponding to the values of these polynomials \cite{Bose}.
These are now known as {\bf Krawtchouk matrices}.

\newpage

To appreciate their elementary character, we start with a high school ``cheat sheet'' 
for algebraic identities and code their coefficients into an array: 
$$ 
\begin{array}{lcl}  
        (a+b)^2   &=& a^2+2ab+b^2          \\
       (a+b)(a-b) &=& a^2\phantom{+2ab\,\,\,} - b^2           \\
        (a-b)^2   &=& a^2-2ab+b^2  
\end{array}
\qquad\Rightarrow\qquad
\left[\begin{array}{rrr}
                             1 &  2 &  1 \cr
                             1 &  0 & -1 \cr
                             1 & -2 &  1 \cr \end{array}\right]
$$
One may design a similar arrays for higher degrees.
These arrays, transposed, define Krawtchouk matrices. 
To simplify the expressions, replace $a=1$ and $b=t$.

\begin{definition}\rm
The $n^{\mathrm{th}}$-order Krawtchouk  matrix $K^{(n)}$ is an
integer $(n\!+\!1)\!\times\!(n\!+\!1)$ matrix,
the entries of which are determined by the expansion:
\begin{equation}
\label{eq:genkraw}
    (1+t)^{n-q} \; (1-t)^q = \sum_{p=0}^{n} \ K^{(n)}_{pq} \, t^p \,.
\end{equation}
The left-hand-side,
$
              G(t)= (1+t)^{n-q}\; (1-t)^q
$
is the {\it generating function} for the entries of the $q^{\mathrm{th}}$ column of $K^{(n)}$.
\end{definition}

We will also use notation with the order $n$ set above $K$, or even omitted if the context allows:
$$
K^{(n)} \ \equiv \ \ontop{n}K \ \equiv \ K
$$
Here are the first few Krawtchouk  matrices:
{\small
$$
\label{eq:krav2}
\ontop{0}K=\left[\begin{array}{rr}{ 1 }\end{array}\right] 
\quad
\ontop{1}K=\left[        
  \begin{array}{rr}
             1 &  1 \cr
             1 & -1 \cr 
             \end{array}\right]  
\quad
\ontop{2}K= \left[\begin{array}{rrr}  1 &  1 &  1 \cr
                             2 &  0 & -2 \cr
                             1 & -1 &  1 \cr \end{array}\right]
\quad
\ontop{3}K=
       \left[\begin{array}{rrrr}
                      1 &  1 &  1  &  1 \cr
                      3 &  1 & -1  & -3 \cr
                      3 & -1 & -1  &  3 \cr
                      1 & -1 &  1  & -1 \cr 
             \end{array}\right]
$$}
More examples can be found in Appendix A, Table 1. 
Expanding (\ref{eq:genkraw}) gives the explicit expression for the matrix entries in terms of binomials:
\begin{equation}
\label{eq:bb}
 \ontop n K_{pq}= \sum_{k} (-1)^k {q \choose k}  {n-q \choose p-k}     \,.
\end{equation}
The generating function may also be presented as
 a product of $n$ terms of the type $(1+\sigma t)$, where for each term, 
sigma is chosen from $\{-1,1\}$.
Expanding, we get: 
\begin{equation}
\label{eq:sigma}
\prod_{\sigma\in\mathbb Z_2^n} (1+\sigma_it)^n
\ = \ 1 + t\,\sum_i \sigma_i + t^2\, \sum_{i,j} \sigma_i\sigma_j + \ldots
\end{equation}
The coefficients are the elementary symmetric functions in $\sigma_i$'s.
\\

One amazing property of Krawtchouk matrices is that their squares are proportional to  the identity matrix:
$$
\ontop n K^2 = 2^n \,I
$$
A simple proof of this fact is in Section \ref{sec:quaternions}.
This property suggests applying Krawtchouk matrices as involutive transforms  
for integer sequences.

By multiplying the columns of the $n$-th Krawtchouk matrix by the corresponding binomial coefficient,
one obtains a {\bf symmetric} Krawtchouk matrix \cite{FF}:
$$
\ontop n K^{symm}_{pq}  \ = \ \ontop n K_{pq} \; {n\choose q}\,.
$$
(see Appendix B).

Yet another characterization of Krawtchouk matrices relates them to Hadamard-Sylvester matrices.
Note that the second matrix coincides with the $2\!\times\!2$ fundamental Sylvester-Hadamard matrix, $K^{(2)}=H$,
so effectively used in quantum computing for preparing the universal entangled states \cite{Lomonaco}.
It turns out that Krawtchouk matrices may be viewed as the symmetric tensor products
of the elementary Hadamard matrix \cite{FK01}:
$$
\ontop nK \ = \ H^{\odot n}
$$ 
where $\odot$ is the symmetric tensor product and the right side is the $n$-tensor power, namely  
$H\odot H\odot \ldots \odot H$. 
\\

In the following sections we review a number of situations that manifest the Ber\-nou\-lli-type random walk. 
The ``master equation'' of the Ehrenfest model, which unifies these examples, provides extended solutions:  Krawtchouk matrices. 
They recover the usual binomial solutions in the first column but also introduce additional ``mystic'' non-physical solutions as the remaining columns. 
We present interpretations of these entries in Sections 5 and 6  (topological and discrete version of Feynman-like sum over paths).  
The last part of the paper clarifies the connections of the Krawtchouk matrices with split quaternions and $SL(2,\mathbb R)$.
The tensor extension of the action of $SL(2,\mathbb R)$ explains the ``master equation'' and provides some other identities, 
including extraction of the eigenvectors of the Krawtchouk matrices.
(The action of $SU(2)$ in the context of Krawtchouk polynomials was also treated in \cite{Koo}.)

Krawtchouk matrices make also a natural appearance in the theory of linear codes and Hamming schemes \cite{Koo, Lev} 
but this subject goes beyond the scope of this presentation.  
Nevertheless, a short Appendix D presents the geometric content of this feature.

\section{Counting, controlled and erratic} 
\label{sec:examples}

Figure (\ref{fig:examples}) contains a number of simple situations that all reduce to the same mathematical concept. 
\begin{figure}[H]
\centering
\includegraphics[scale=.9]{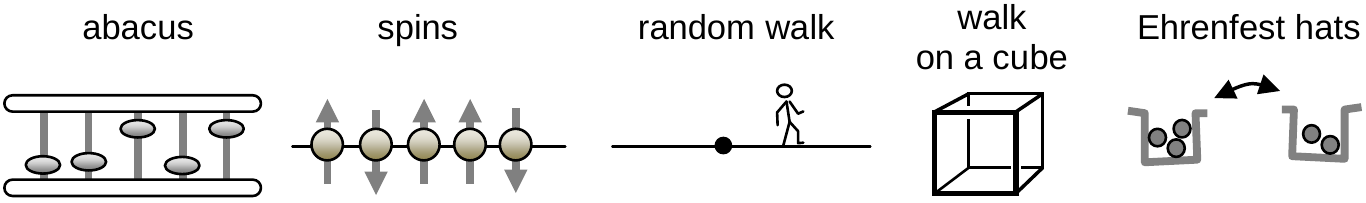} 
\caption{Examples of Bernoulli walk experiments}
\label{fig:examples}
\end{figure}

\noindent
{\bf 1. Abacus.} Classical computing may be understood as the manipulation of abacus, Fig. \ref{fig:examples}.
The picture represents a configuration of the beads.
One bead lives in $\mathbb Z_2\equiv \{0,1\}$.
The configuration space of $n$ beads is
the direct product of $n$ copies of the individual state spaces
$$
       \mathbb Z_2^n = \mathbb Z_2\oplus \mathbb Z_2\oplus\ldots\oplus \mathbb Z_2
$$
(a discrete $n$-cube).
Now, classical computing understood as a manipulation of the beads translates into 
a controlled walk on the cube.
The invertible endomorphisms (automorphisms) will correspond to reversible classical computer gates.
A probability-theoretic question to be asked is: if Borel's monkey, instead of a typewriter were operating a binary abacus,
what is its expected state after a long time.   
In particular, how often would $p$ beads end in the left position.
\\[7pt]
{\bf 2. Spins.}  The second image in Figure \ref{fig:examples}  shows 
a diagrammatic representation of a system of electron spins.%
\footnote{We assume a convention of integer spins $\pm1$ instead of  $\pm\frac{1}{2}$.} 
Spins may be oriented  "up" or "down", and the orientation of each may be controlled (switched)
using, say, magnetic field. 
Any spin arrangement may be exactly copied to a bead arrangement on abacus,
thus, in essence, we have the same counting device. 

Figure  \ref{fig:spins} shows all possible arrangements of three spins.
Note the step-2 descending arithmetic sequence of the resulting total spins, the sums of the individual spins: 
$3$, $1$, $-1$, $-3$, as reported at the top row.
The number of different configurations giving the same total spin forms 
a sequence of binomial coefficients, here (1, 3, 3, 1).
All states of the system (here just three spins)
make $2^n$ (here 8) configurations.

\begin{figure}[H]
\centering
\includegraphics[scale=.9]{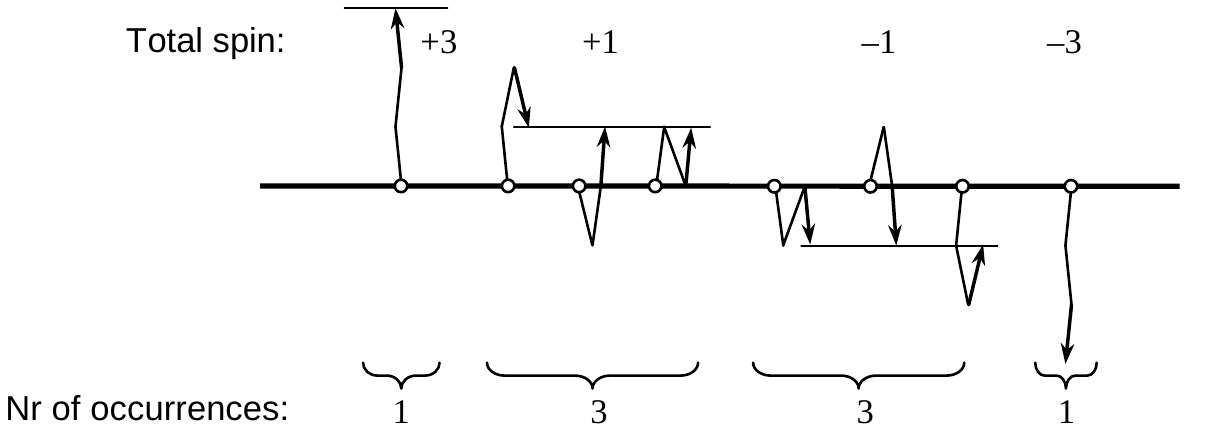} 
\caption{Spin arrangements}
\label{fig:spins}
\end{figure}

\noindent
{\bf 3. Random walk.}  Next situation concerns the classical ``drunkard's walk'', 
or 1-dimensional symmetric random walk of step 1.
A walker makes a random selection between making a unit step  in the left vs right direction (with equal probability).
The question of the probability of finding the walker after $n$ steps at position $x\in \mathbb Z$ 
leads to the binomial distribution.
This is obviously equivalent to the spin arrangement of the previous example:
simply locate the total spins of Figure \ref{fig:examples} horizontally
and reinterpret Figure \ref{fig:spins}.
\\[7pt]
{\bf 4. Ant on a cube.} An ant walks along the edges of an $n$-dimensional cube, choosing at every vertex the next edge randomly.
If one thinks of the cube as the $n$-dimensional linear space over $\mathbb Z_2=\{0,1\}$,
the analogy to the previous examples is obvious.
Every vertex is equivalent to a configuration.
A move along an edge corresponds to a single spin flip.
The total spin of Example 2, for instance, translates into the discrete graph-theoretic distance of the vertex from the origin.
In particular, the numbers of the vertices in the consecutive planes perpendicular to the main diagonal of the $n$-cube 
are binary coefficients $n\choose i$.  
The cube realization is the most universal exposition of the problem.
\\\\ 
{\bf 5. Ehrenfest hat problem.}  Two hats contain a total of $n$ balls.
One is drawn at random and transfered to the other hat.
A question one may ask is the distribution of the balls after a long run of such experiment.
This situation is again just another version of our archetypical problem.
It will be discussed in detail in Section \ref{sec:hat}.
\\[7pt]
{\bf 6. Coins.}  Just for completeness, let us add another example: $n$ coins on a table.%
One step of a process consists of turning a random coin up-side down. 
Question: what is the limit number of heads-up coins.
\footnote{As an exposition aimed at its simplicity, the Ehrenfest hat problem has a drawback: 
how does one choose a {\it random} ball without choosing first a hat?
The coin version avoids the problem.}
\\

Note the dual character of each example: 
when controlled, it becomes a counting device.
When left at random behavior,
it models Bernoulli random walk with binomials distribution of cluster of states,
the claster defined by a certain distance (weight) function.

\section{What is quantum computing?} 

We open with a general comment on the nature of quantum computing.

\subsection{The concept of splay}

Given a linear space $V$, one may construct a new space $\mathcal A$ by reinterpreting 
each vector of $V$  as a basis vector of the new space.
In particular,  $\dim \mathcal A = \card V$.
We shall call the new object the splay%
\footnote{The term is to relate to the motion of ``splaying fingers.''  C.f., French: {\it \'esciter'}, Polish: {rozcapierzy\'c}.}
of $V$ and denote $\splay_{\mathbb F}V$.
\\
\\
{\bf Definition:} 
Let $V$ be a vector space (or a module).
We say that a vector space $\mathcal A$ over a field $\mathbb F$ is a {\bf splay} of $V$, 
denoted $\splay_{\mathbb F} V$, 
if there is a map
$$
V\to\mathcal A:\  v\mapsto \tilde v
$$
such that 
$$
 \rmspan\{\; \tilde v \;\big|\; v\in V\;\}   \ = \ \mathcal A \equiv \splay_{\mathbb F}\, V 
$$
and for any $v,w\in V$
$$
v \not= w \qquad \Rightarrow\qquad \tilde v \bot \tilde w \,.
$$
An algebra algebra $\mathcal A$ is an {\bf algebraic splay} of a vector space $V$ if its underlying space  
is the splay of $V$,  
and the product in $\mathcal A$ agrees with linear structure of $V$ in the sense 
$$
v, w \in V \qquad \widetilde v\; \widetilde w \ \dot= \  \widetilde{v+w}
$$
where $\dot = $ means ``equal up to a scalar factor''.
It follows that  the image of the zero vector in $V$ is proportional to the algebra identity, ${\widetilde 0} = 1$.
\\

We encounter splays quite often.
For instance, the algebra of quaternions is a splay of a square,  $\mathbb H \cong\splay \mathbb Z_2^2$.
Similarly for octonions $\mathbb O \cong\splay \mathbb Z_2^3$, 
and for complex numbers $\mathbb C \cong\splay \mathbb Z_2^1$. 
Any Clifford algebra of a space $\mathbb R^{p,q}$ may be viewed a  splay of $\mathbb Z_2^{p+q}$.
Clearly, to make the algebraic splay well-defined, one needs to specify a function $f:V\times V\to \mathbb F$ to
and set $\widetilde v\;\widetilde w = f(v,w)\,(\widetilde{v+w})$.

\subsection{What is quantum computing}

The question set in the title is usually answered with many words.
Here is a proposition for a very short answer:
Quantum computing is a replacement of the cube 
$$
            \mathbb Z_2^n \ = \ \mathbb Z_2\oplus \mathbb Z_2\oplus\ldots\oplus \mathbb Z_2
$$
by its complex splay
$$
     \splay_{\mathbb C} \,\mathbb Z_2^n \ = \ \mathbb C^2 \otimes  \mathbb C^2 \otimes \ldots \mathbb C^2 
\ \cong \            \mathbb (C^2)^{\otimes n} 
\ \cong \   \mathbb C^{2^n}  
$$
The first is a discrete cube, the second may be reduced to sphere.
Classical computing is a {\bf walk} on the $n$-cube, the quantum computing is {\bf turning} the sphere
in $\mathbb C^{2n}$.
In classical computing we operate by discrete rearrangements of beads,
in quantum computing, we may turn the states (points on the sphere)
to arbitrary positions that no longer have classical interpretations, 
and have to be interpreted as superpositions and entanglements.

\begin{figure}
\centering
\includegraphics[scale=.9]{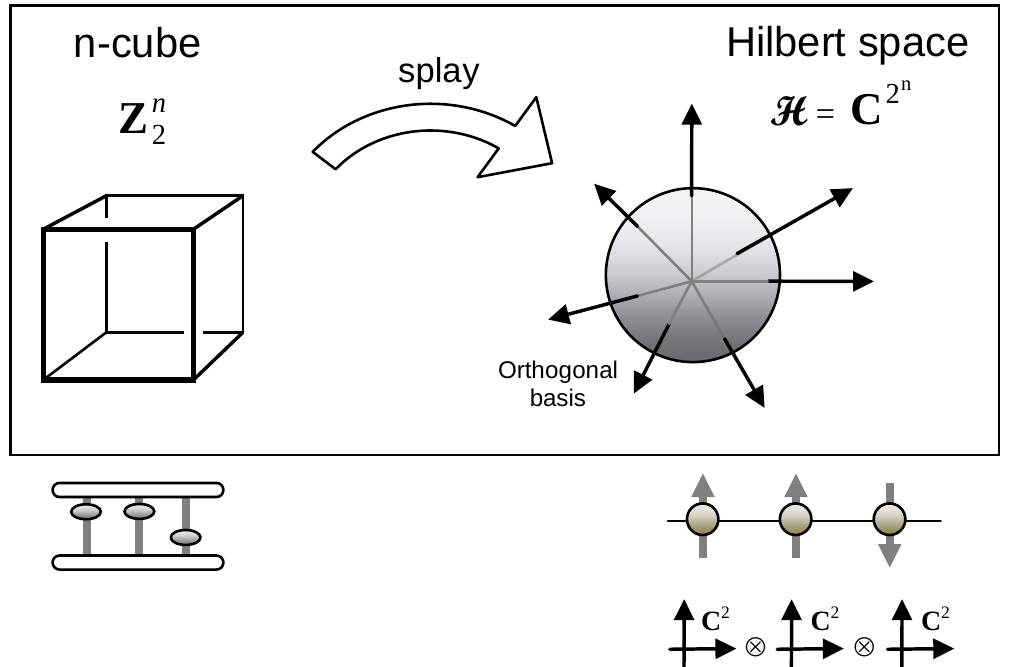}
\caption{Quantum computing as a splay of classical computing}
\label{fig:splay}
\end{figure}

Usually, computers are presented as embodiment of the Turing universal machine.
A better mathematical metaphor seems however a binary abacus with a single bead on each rung.
The array of electron spins in Fig. \ref{fig:splay}, when viewed classically, is equivalent to beads of abacus.
But due to their quantum behavior, the two states "up" and "down"
are best represented by two vectors in $\mathbb C^2$
$$
\hbox{``up''}  =  \begin{bmatrix} 1\\0\end{bmatrix},\qquad
\hbox{``down''}  =  \begin{bmatrix} 0\\1\end{bmatrix}
$$
Their position may be controlled by electromagnetic interactions,
which mathematically are represented by unitary operators in $U(2)$.  
The ensemble of $n$ electrons form the tensor product of the individual spaces,
which may be controlled with unitary operators from $U(2^n)$.

\section{Ehrenfest Urn problem: beyond Kac's solution}
\label{sec:hat}

Ehrenfest designed his two-hat experiment to illustrate certain statistical behavior of particles.
Let us slightly reformulate it  into an equivalent of Example 6.
We have {\it one} hat with $n$ balls, some being lead and some gold.  
At each time $k\in\mathbb N$, 
a ball is drawn at random, changed by a Midas-like touch into the opposite state
(gold $\leftrightarrow$ lead) and placed back in the hat.
One of the questions is of course about the distribution of states 
as time goes to infinity, $k\to \infty$.

\begin{figure}[h]
\centering
\includegraphics[scale=.9]{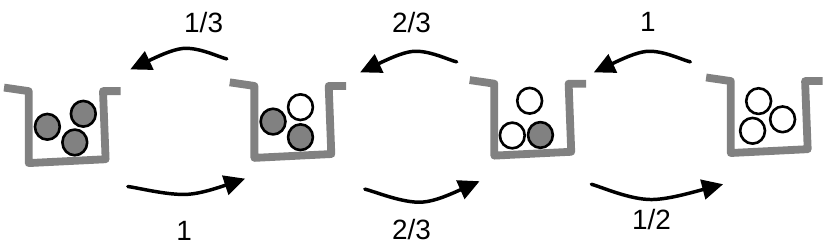}
\caption{Four Markov states of Ehrenfest hat with three balls. Dark shade denotes gold}
\label{fig:markov}
\end{figure}

Mark Kac showed an elegant elementary solution using the Markov chain technique. 
The states of the model are represented by basis vectors in $\mathbb{R}^{n+1}$, 
namely the state of $k$ gold balls in the hat by:  
$$
\mathbf{v}_k = [\; 0 \ 0 \  \cdots  \ 1 \ \cdots \ 0 \;]^T
$$
with ``1'' at the $k$-th position.
In the case of $n=3$, we have 4 pure states:
{\small
$$
\begin{array}{l} \hbox{\small \textsf{0 gold balls}}\\ 
                 \hbox{\small \textsf{3 lead balls}}  \end{array}
   = \left[ \begin{array}{c} 1 \\ 0 \\ 0 \\ 0  \end{array}\right]
\quad
\begin{array}{l} \hbox{\small\textsf{1 gold}}\\ 
                 \hbox{\small\textsf{2 lead}} \end{array}
   = \left[ \begin{array}{c} 0 \\ 1 \\ 0 \\ 0 \end{array}\right]
\quad
\begin{array}{l} \hbox{\small\textsf{2 gold}}\\ 
                 \hbox{\small\textsf{1 lead}} \end{array}
   = \left[ \begin{array}{c} 0 \\ 0 \\ 1 \\ 0 \end{array}\right]
\quad
\begin{array}{l} \hbox{\small\textsf{3 gold balls}}\\ 
                 \hbox{\small\textsf{0 lead balls}}  \end{array}
   =  \left[ \begin{array}{c} 0 \\ 0 \\ 0 \\ 1  \end{array}\right]
$$}
Figure \ref{fig:markov} shows the corresponding digraph 
with the edges labeled by the transition probabilities.
The stochastic matrix of the state transition 
(a single ball drawing) is: 
$$
  \left[\begin{array}{rrrr} 
                0 & \frac{1}{3} & 0 & 0 \cr
                1 & 0 & \frac{2}{3} & 0 \cr
                0 & \frac{2}{3} & 0 & 1 \cr
                0 & 0 & {1\over3} & 0 \cr \end{array}\right]
\quad = \quad 
\frac{1}{3}\,
  \left[\begin{array}{rrrr} 
                0 & 1 & 0 & 0 \cr
                3 & 0 & 2 & 0 \cr
                0 & 2 & 0 & 3 \cr
                0 & 0 & 1 & 0 \cr \end{array}\right]
\quad = \quad\frac{1}{3}\, M^{(3)} \ .
$$ 
We define the $n$-th {\bf Kac matrix} as an integer matrix with two off-diagonals in arithmetic progression 
$1,2,3,...,n$,  descending and ascending, respectively:
\begin{equation}
\label{eq:Kac}
M^{(n)} =
{\smalll
  \left[\begin{array}{ccccccc} 
    0  & 1       &    &  && & \\
    N & 0       & 2   &  && & \\
     & N-1     & 0   & 3 &&& \\
     &        & N-2 & 0 &4& & \\ 
              &&&&\ddots&\ddots&\\
     &        &    &  &\ddots&0& N \\ 
     &        &    &  &&1& 0  
\end{array}\right]
}
\end{equation}
(zeros at unoccupied positions). 
Thus the problem boils down to finding the stable state,
the eigenvector of the Kac matrix with the eigenvalue $n$,
(not 1, since the actual stochastic matrix is $\frac{1}{n}A)$. 
\begin{equation}
\label{eq:eigen}
       M\mathbf v=\lambda \cdot \mathbf  v
\end{equation}
%
It is easy to guess that the eigenvalue equation is satisfied by the binomial distribution, i.e.,  the vector $v$ with components $v^i={n\choose i}$. 
For $n=3$:
$$
\frac{1}{3}
\left[\begin{array}{rrrr} 
         0 & 1 & 0 & 0 \cr
         3 & 0 & 2 & 0 \cr
         0 & 2 & 0 & 3 \cr
         0 & 0 & 1 & 0 \cr \end{array}\right]
\;
 \left[\begin{array}{r} 
                1  \cr
                3  \cr
                3  \cr
                1  \cr \end{array}\right]
= \left[\begin{array}{r} 
                1  \cr
                3  \cr
                3  \cr
                1  \cr \end{array}\right]
$$
Rescaling the vector by $2^{-n}$ reproduces the expected binomial discrete probability distribution,
the expected state after a along run of the process.
\\

From the probabilistic point of view one stops right here.
But --- from the algebraic point of view --- it is not the whole story yet.
There are other formal solutions to equation (\ref{eq:eigen}),
for instance, continuing the case of $n=3$, 
$$
\frac{1}{3}
\left[\begin{array}{rrrr} 
         0 & 1 & 0 & 0 \cr
         3 & 0 & 2 & 0 \cr
         0 & 2 & 0 & 3 \cr
         0 & 0 & 1 & 0 \cr \end{array}\right]
\;
 \left[\begin{array}{r} 
                1  \cr
                1  \cr
                -1  \cr
                -1  \cr \end{array}\right]
= \frac{1}{3}\left[\begin{array}{r} 
                1  \cr
                1  \cr
                -1  \cr
                -1  \cr \end{array}\right]
$$
and similarly for $[1,-1,-1,1]^T$ and $[1,-3,3,-1]^T$ with eigenvalues $-1$ and $ -3$, respectively.  
These may be gathered into a single matrix-like equation:
\begin{center}
Kac matrix ~~~~~~~~~~~eigenvectors ~~~~~~~~~~~~~~~~~~ eigenvectors ~~~~~~~~~~ eigenvalues~~
\end{center}
{\small
$$
  \underbrace{
   \left[\begin{array}{rrrr} 
                0 & 1 & 0 & 0 \cr
                3 & 0 & 2 & 0 \cr
                0 & 2 & 0 & 3 \cr
                0 & 0 & 1 & 0 \cr \end{array}\right]   }_{\mathlarger M}
  \ 
\underbrace{
\left[\begin{array}{r|r|r|r} 
     1 &  1 &  1 &  1 \cr
     3 &  1 & -1 & -3 \cr
     3 & -1 & -1 &  3 \cr
     1 & -1 &  1 & -1 \cr \end{array}\right]     }_{\mathlarger K}
\ = \ 
\underbrace{
  \left[\begin{array}{r|r|r|r} 
     1 &  1 &  1 &  1 \cr
     3 &  1 & -1 & -3 \cr
     3 & -1 & -1 &  3 \cr
     1 & -1 &  1 & -1 \cr \end{array}\right] }_{\mathlarger K}
\  
\underbrace{
\left[\begin{array}{rrrr} 
          3 & 0 & 0 & 0 \cr
          0 & 1 & 0 & 0 \cr
          0 & 0 & -1 & 0 \cr
          0 & 0 & 0 & -3 \cr \end{array}\right]  }_{\mathlarger\Lambda}
$$

}

\begin{theorem} [{\cite{FK01}}]
The spectral solution to the Kac matrix (\ref{eq:Kac}) involves Krawtchouk matrices, 
namely, the {\it collective solution} 
to the eigenvalue problem $Av=\lambda v$ is   
\begin{equation}
\label{eq:master}
\mathlarger { M K =  K \Lambda }
\end{equation}
where the ``eigenmatrix" $K$ is the Krawtchouk matrix of order $n$ and  
$\Lambda$ is the $(n\!+\!1)\!\times\!(n\!+\!1)$ diagonal integral matrix with entries
$\Lambda_{ii}=N-2i$, that is, 
$$
\Lambda^{(n)} =
  \left[\begin{array}{ccccccc} 
      N &     &    &  &  & \\
     & N-2 &    & & & \\
     &     & N-4&   &&\\
     &     &    & \ddots &&\\ 
     &     &   &      &2-N  &\\ 
     &     &    &     &    & -N  
\end{array}\right]
$$ 
(unoccupied places imply zeros).
\end{theorem}

In summary, Krawtchouk matrices are {\bf extended} solutions to the Ehrenfest hat problem
and---by equivalence---to any of the other problems listed in Section \ref{sec:examples}. 
Equation (\ref{eq:master}) is our {\bf master equation} to be explored for gaining an insight into Krawtchouk matrices.
However, only the first column has the clear meaning.
The other columns are  unrealistic ``shadow solutions'' that  
do not seem to have direct probabilistic interpretation.
Their entries 
not only include negative terms, but also 
sum up to zero, hence cannot be normalized to represent probabilities.
Search for their interpretation is the motivation for the rest of the paper.
We shall present: 
\\[-20pt] 
\begin{enumerate}
\item
A toy model of hypothetical particles ``twistons''.\\[-19pt]
\item
An analogon  of the Feynman path integral formalism of Quantum Mechanics.\\[-19pt]
\item 
An analogon of the formalism quantum computing formalism: results of the paper \cite{FK01} but with new interpretation via 
split quaternions.
\end{enumerate}

\section{Topological interpretation of Krawtchouk matrices}

Here we describe a toy model for Krawtchouk matrices.
Imagine a bath of $n$ hypothetical particles represented by closed strips,  
call them ``twistons''. 
They can assume  two states: orientable (topologically equivalent to cylinders)
and non-orientable (equivalent to M\"obius strips).
We associate with them energies $E=\pm 1$, as shown in Figure \ref{fig:twistons}, left. 
Twistons interact in the sense that 
two strips, when in contact can form a single strip, Figure \ref{fig:twistons}, right.
For instance, two M\"obius strips will become an orientable strip, while a M\"obius strip with a straight strip
will result in a M\"obius strip.

Suppose one wants to calculate the {\bf interaction energy} of the whole bath of $n$ twistons.
\begin{figure}[h]
\centering
\includegraphics[scale=.7]{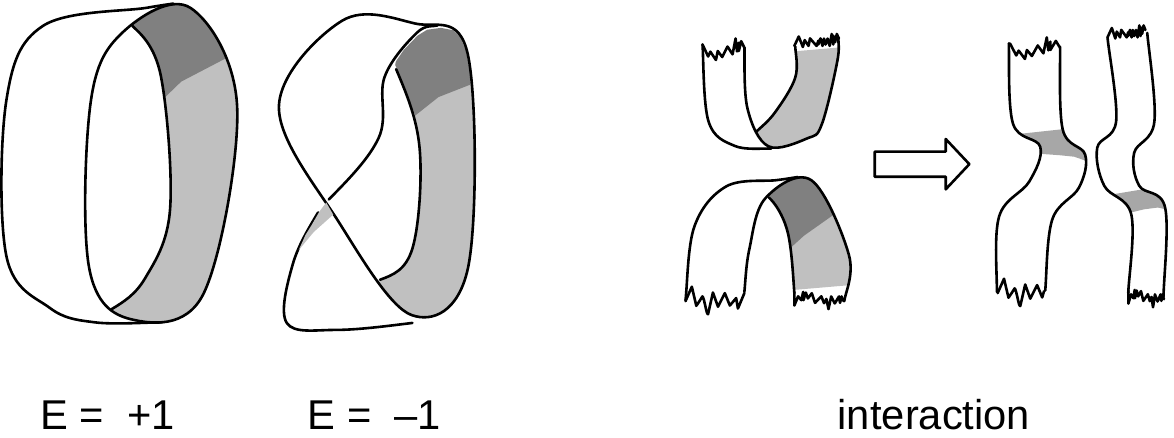}
\caption{two types of twistons and interaction}
\label{fig:twistons}
\end{figure}

The 1-energy is defined as the total energy of the system, the sum of $E_i$'s.
The 2-energy, the total pair-wise interaction energy, 
is the sum of the energy products of all pairs of twists.  
In other words, $\sum_{i\neq j} E_iE_j$.
Similarly, the triple interaction energy will be the sum of the triple products,
etc.
(We set the 0-energy to be 1.)
But this is the same as the decomposition of the generating function to the simple 
elementary functions in terms of $\sigma$'s, as in Equation (\ref{eq:sigma}).
These interaction potentials 
will depend on the number $q$ of non-oriented twistons among the $n$.
Hence the result:
Krawtchouk matrix lists all interaction energies for the bath of twistons.
$$
\ontop{n}K_{pq} = \left\{           \hbox{the {\it p}-interaction energy of the bath}
                                      \atop \hbox{of {\it n} twistons with {\it q} M\"obius strips}               \right\}
$$
In the bath of $n$ twistons with exactly $q$ untwisted bands,
the total energy of $p$-wise interactions coincides
with the $p$-th entry of the $q$-th column of the $n$-th Krawtchouk matrix.

\section{Feynman sum over paths interpretation}

The Galton machine is a triangular array of pins from which falling 
balls are deflected randomly to the left or right, and collected in boxes at the bottom
(Figure \ref{fig:Galton}).
As the number of balls increases, the distribution at the bottom becomes 
close to the binomial distribution (central limit theorem).

\begin{figure}[h]
\centering
\includegraphics[scale=.7]{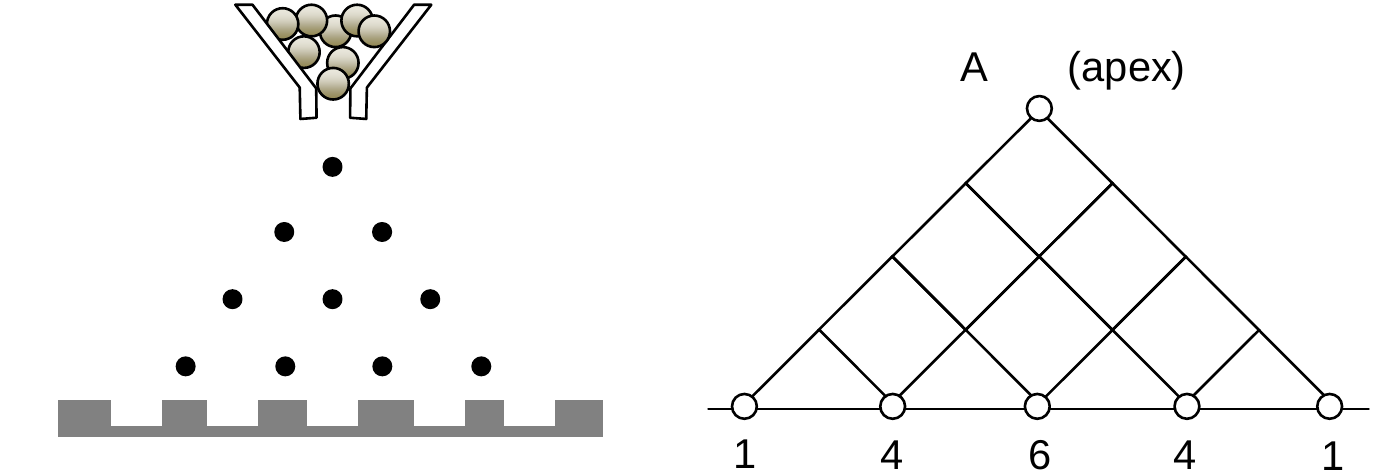}
\caption{Galton machine}
\label{fig:Galton}
\end{figure}

One of the formulations of Quantum Mechanics is the so-called Feynman path integral formalism.
One calculates the probability of the transition from state $A$ to $B$ by 
associating to each path $c$ from $A$ to $B$  a certain phase (a complex number)
and then summing up (integrating) these phases over all possible paths.
With ignoring some constants, 
$$
\hbox{prob}(A,B) \ = \ 
\int\limits_{\hbox{\sf \smalll all paths}} \!\! e^{i\int_c {L\,dt}}
$$
where $L$ is the Lagrangian  of the system, and the integral symbol denotes the sum over paths from $A$ to $B$, see e.g., \cite{Fey}.
This integral involves a rather sophisticated machinery and therefore does not belong to the gear of popular tools
outside of quantum physics.
But the core of the idea is simple and should be  known better.
In particular, it may be translated to discrete systems, where 
it becomes clear that it is based on a duality:
$$
          \left\{           \hbox{combinatorics}
                     \atop \hbox{of paths}        \right\}
\qquad\longleftrightarrow\qquad
 \left\{           \hbox{tree of conditional}
                     \atop \hbox{probabilities}    \right\}
$$

To see this duality, let us return to the Galton machine, 
represented it by a grid, as in Figure \ref{fig:Galton}, right. 
Label the positions at the bottom line by $p=0,1,...,n$, where $n$ is the hight of the machine.
The number of descending paths along the grid from the most top position, apex $\mathbf A$, to the $p$-th vertex at the base is 
is equal to a binomial coefficient $n\choose p$.
Hence we reconstructed the first column of the corresponding Krawtchouk matrix.%
\footnote{This is in a sense the classical version of Feynman path integral.
It is an equivalence principle between the probability }
Here is how we get the other columns.

\begin{figure}[H]
\centering
\includegraphics[scale=.5]{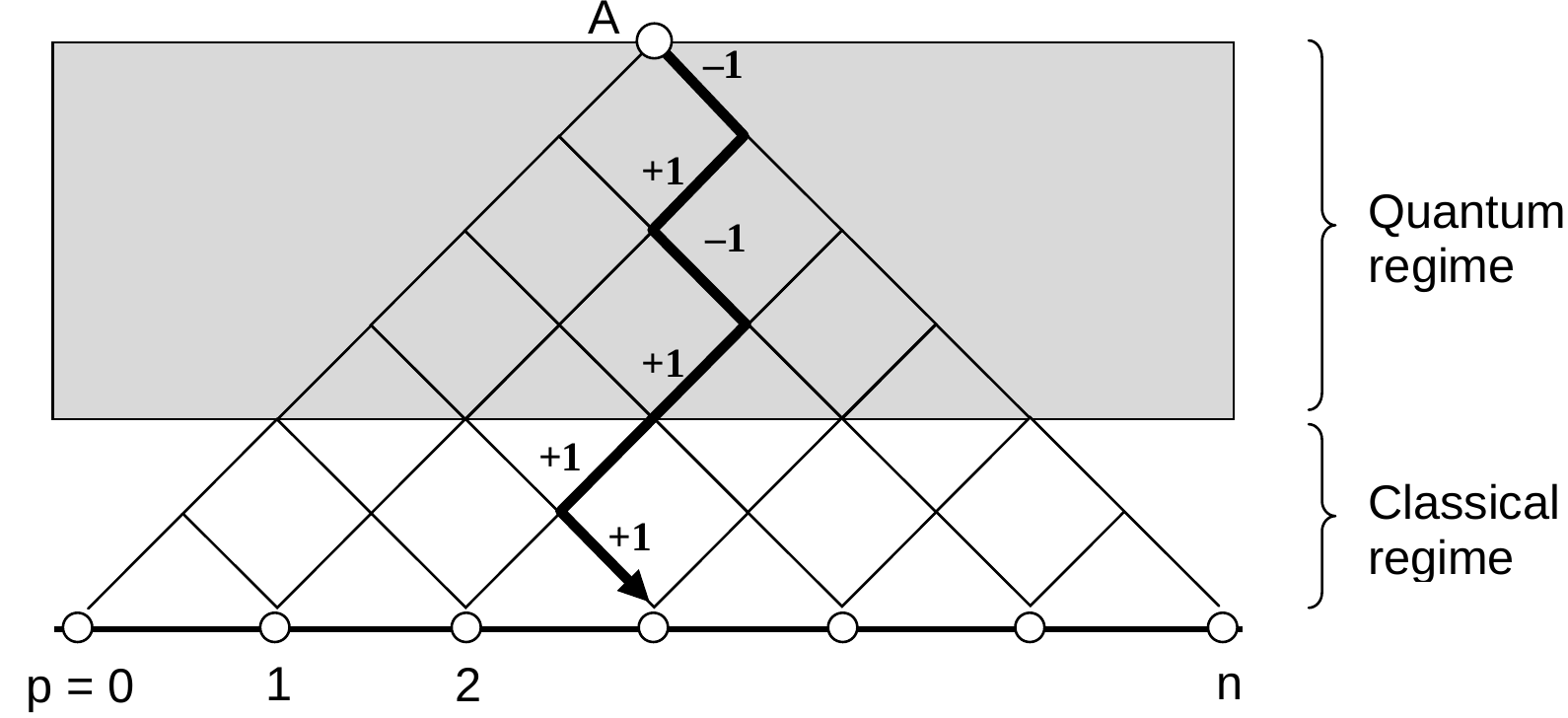}
\caption{Feynman path interpretation of Krawtchouk matrix.  
This path contributes $(-1)\cdot 1\cdot(-1)\cdot1\cdot 1\cdot 1 = 1$
(as read along the path from the top down)
to the value at the point on the bottom line.}
\label{fig:Feynman}
\end{figure}

Divide the Galton board into two parts: one of ``quantum regime'' that covers the first $q$ steps
(shaded region in Figure \ref{fig:Feynman}) and one of the ``classical regime'', covering the remaining levels. 
A descending path from the apex $A$ down to the bottom line consists of $n$ steps,
each step contributing a factor to the path's weight.
The total weight of a path is the product of these step factors.
In the quantum region, every step going to the right contributes a factor $-1$,
and going to the left, $+1$.
In the classical region, all factors are $+1$. 
For instance in Figure \ref{fig:paths}, the bundle of paths gives
$$
(q=2,\ p=3)\qquad
\mathbf 1 \!\cdot\! ({-\mathbf 1})  \!\cdot\! 1 \!\cdot\! 1  \ + \
({-\mathbf 1})  \!\cdot\! \mathbf 1 \!\cdot\! 1 \!\cdot\! 1 \ + \ 
\mathbf 1  \!\cdot\! \mathbf 1 \!\cdot\! 1 \!\cdot\! 1  \ + \ 
\mathbf 1  \!\cdot\! \mathbf 1 \!\cdot\! 1 \!\cdot\! 1 \ = \ 0
$$
(the contributions in the quantum region are in bold font).
\\

Here is the result:  {\bf This machine generates Krawtchouk matrices.}
For the thickness $q$ of the ``quantum region'', 
the distribution of the path sums at the base coincides with the entries of the $q$-th column of $K$.

\begin{figure}[h]
\centering
\includegraphics[scale=.7]{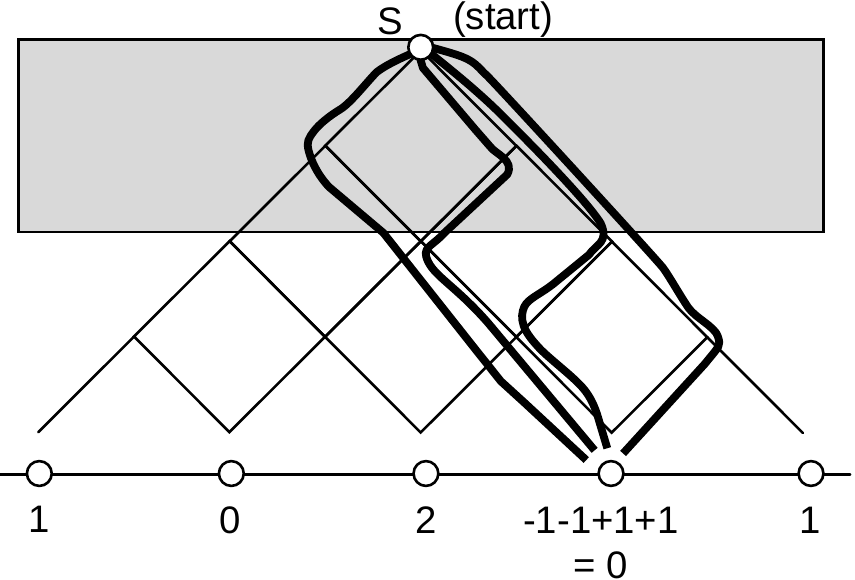}
\caption{Every path contributes to the sum over paths.  
The phases may be constructive or destructive. 
Here is are the resulting sums for quantum regime of size $q=2$.}
\label{fig:paths}
\end{figure}

Let us formalize it.
Every path, as a string of binary choices: $L=$ left, and $R=$ right, 
is coded by a word in the word algebra over alphabet $\{R,L\}$.
$$
w\in \{L,R\}^*, \quad w=w_1w_2\ldots w_n \qquad w_i\in \{L,R\}
$$ 
Define the factor function of the individual letters:  
$$
\lambda (L) = 1 \qquad \lambda(R) = -1\,.
$$The weight contributed by a particular path is given by
$$
\hbox{weight}_q (w) = \prod_{i=1}^q \lambda(w_i)
$$
The product is cut at $q$, which means that the ``quantum regime'' is turned on for the first $q$ steps only. 
Denote the set of all descending paths from the top vertex (apex) to position $p$
in the base line by $C(p)$.

\begin{proposition}
\label{thm:Feynman}
If the quantum regime is in effect in the first $q$ steps,  
the sum over paths from the apex  to position $p$ on the base line coincides with the values in the Krawtchouk matrix: 
\begin{equation}
\int\limits_{w\in C(p)} \prod_i^{k}\lambda(w_i)  \ = \ \ontop n K_{pq}
\label{eq:Feynman1}
\end{equation}
(We use ``$\int$'' instead of ``$\sum$'' to emphasize the analogy to quantum physics.)
\end{proposition}
\noindent
{\bf Proof:} 
The set of paths $C(p)$ from the apex down to the $p$-th position in a lattice of hight $n$
is in the one-to-one correspondence with the words in $\{R,L\}^*$ of length $n$,  
such that the number of $R$'s in $w$ is exactly $p$, hence $C(p) \subset \{L,R\,\}^{[n]}$.
Read words from left to right.
Word $w$ consists of two segments, the first, ``quantum'', with $q$ letters, followed by the ``classical'' of length $n-q$.
If the number of $R$'s in the first segment is $k$, the word will have weight $(-1)^k$.
There are $q\choose k$ such word segments, appended by $n-q\choose p-k$ possible ``classical'' segments 
(the remaining $n\!-\!k$ letters $R$ must be distributed among the $(n-q)$ positions).
Thus the sum of the weights over all words in $C(p)$ is
$$
\sum_k  (-1)^k {q \choose k}{n-q\choose p-k}
$$
But this is Formula (\ref{eq:bb}) for the Krawtchouk matrix entries. \QED

\subsection{Generalization to complex phases}

In order to expose better the affinity of Formula (\ref{eq:Feynman1}) to the spirit of the Feynman path integral, 
note that $+1=e^{0\cdot i}$ and $-1 = e^{\pi\cdot i}$. 
We may redefine the phase function in the quantum region to be: 
$$
\varphi(L)= 1\qquad  \varphi (R) =  e^{\varphi i}
$$
where $\varphi$ is a fixed constant phase factor.
Turning this phase contribution for the first $q$ steps gives  
a generalized complex-valued Krawtchouk matrix:
\begin{equation}
K_{pq}(\varphi) \ = \ \int\limits_{w\in C(x)} e^{i\sum_{k=i}^k\lambda(w_i)} 
\end{equation}
As before, the  symbol of integral denotes the sum over all histories from the top to $p$. 
The exponent plays the role of the Lagrangian in the original Feynman integral.  
\\

For $\varphi=\pi$, we get the classical Krawtchouk matrices. 
For $\varphi=\pi/2$, the first $q$ right turns contribute each the factor of $i$.
Here are the corresponding Krawtchouk matrices:
$$ 
K(i)=\left[\begin{array}{cc}   
1&1  \\
1&i
\end{array}\right]
\quad
\ontop 2 K(i)=\left[\begin{array}{ccc}   
1&1 &1  \\
2&1+i &2i  \\
1&i &-1
\end{array}\right]
\quad
\ontop 3 K(i)=\left[\begin{array}{cccc}   
1&1 &1 &1  \\
3&2+i &1+2i &3i \\
3&1+2i &-1+2i &-3 \\
1&i &-1 &-i 
\end{array}\right]
$$
{\small
$$
\ontop 4 K=\left[\begin{array}{ccccc}   
1\!\!&1 \!&1\!&1 \!&1               \!\\
4\!\!&3+i \!&2+2i \!&1+3i \!&4i \!\\
6\!\!&3+3i \!&4i \!&-3+3i\!&-6 \!\\
4\!\!&1+3i \!&-2+2i \!&-3-i\!&-4i \!\\
1\!\!&i \!&-1 \!&-i\! &1 \!
\end{array}\right]
\quad
\ontop 5 K=\left[\begin{array}{cccccc}   
1\!\!  	&1		&1 \!  	&1\!		&1 \! 	&1               \!\\
5\!\!  	&4+i		&3+2i \! 	&2+3i \!   	&1+4i \!	&5i \!\\
10\!\!  	&6+4i &2+6i \!	&-2+6i \!	&-6+4i\!	&-10 \!\\
10\!\! 	&4+6i	&-2+6i \!	&-6+2i \!	&-6-4i\!	&-10i \!\\
5\!\!  	&1+4i	&-3+2i \!	&-3-2i \!	&1-4i\!	&5 \! \\
1\!\!   	&i \!		&-1 \!	&-i\! 		1&&i \! 
\end{array}\right]
$$
}
\begin{figure}[h]
\centering
\includegraphics[scale=.7]{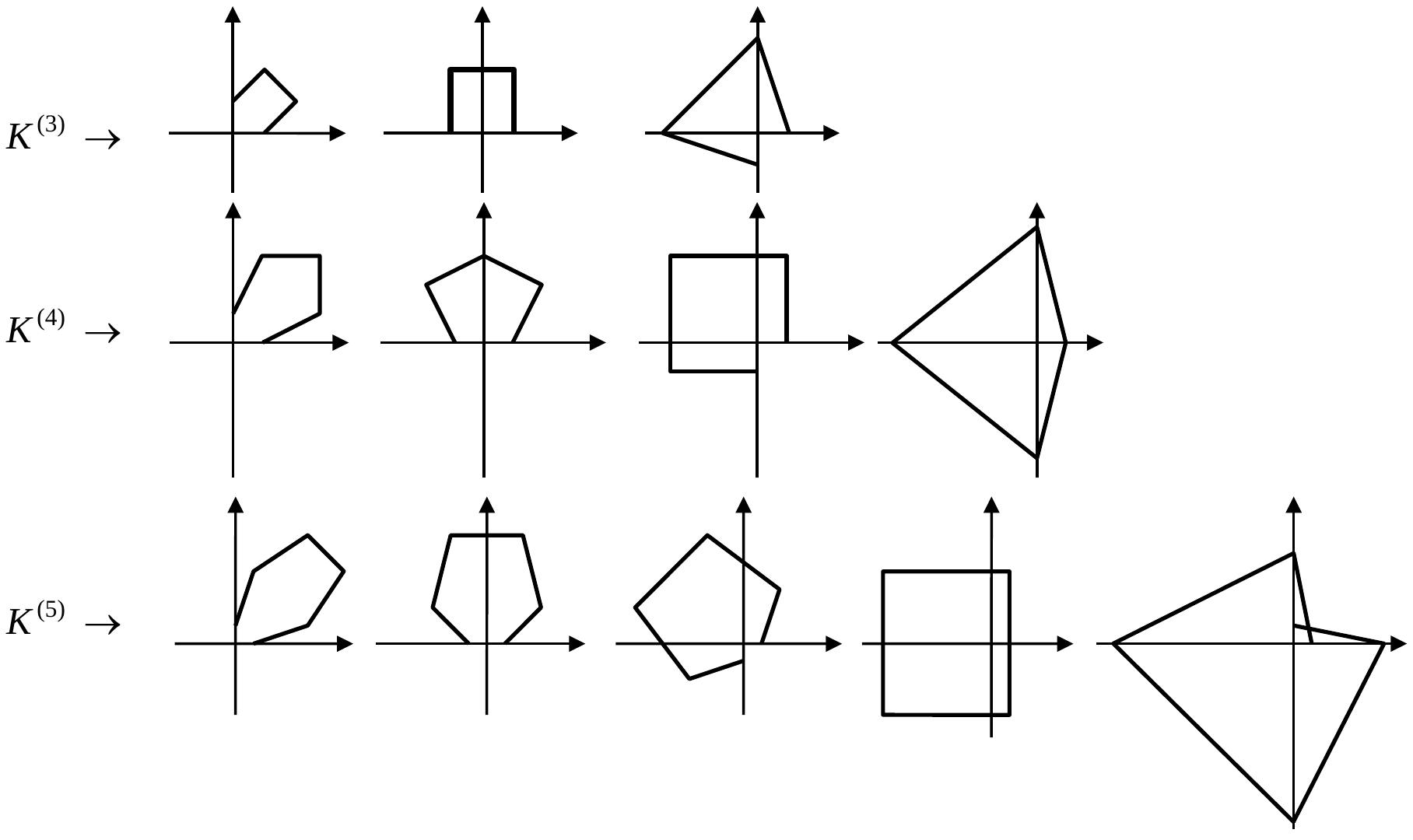}
\caption{The grace of Krawtchouk matrices.  
Columns (starting with the second) of Krawtchouk matrices $K(\varphi=i)$ as paths.}
\label{fig:snakes}
\end{figure} 

The complex-valued version of the Krawtchouk matrices displays their hidden grace.  
In Figure \ref{fig:snakes},   
the entries of each column (except the first) are drawn in the Argand plane with 
segments connecting pairs of the consecutive entries.
Every column becomes a bended snake with a $D2$ symmetry.
A similar type of snake-like figures may be obtained by drawing the rows this way.

\subsection{\large Generalization to phases from a ring}
\label{sec:ring}

Following the analogy with Feynman path integral, we may explore a more general setup, 
letting {\it both} directions, 
$L$ and $R$, contribute a phase in the ``quantum'' region.
Let us go beyond complex numbers and consider a commutative ring $\mathcal R$
and define the phase factor function $\lambda : \{L,R\}\to \mathcal R$ by:
$$
\lambda(v)= \begin{cases} \alpha & \hbox{if} \  v= L\\
                                      \beta & \hbox{if} \  v = R \end{cases}
$$
in the quantum region,  and $\lambda (L)= \lambda(R)=1$ in the classical region.
The weight of a path is, as before, the product of the phases of the individual steps, 
and the total weight at a base vertex is the sum over all paths to this vertex.
Hence we define:
$$
\langle\; \hbox{Apex}\;\big|\; p\;\rangle \quad  \ontop {\hbox{\smalll def}}{\mathlarger =\!\!\mathlarger =\! } \quad
\int_{C(p)} \;\prod_{i=1}^q \lambda(w^i)
$$
where $C(p)$ denotes the set of all descending paths from the top vertex to position $p$ at the base.
\\[-7pt]

On the other hand, define a generalized Krawtchouk matrix with entries from the ring $\mathcal R$:
\begin{equation}
\label{eq:genkrawnew}
    (1+\alpha t)^{n-q} \; (1+\beta t)^q \quad = \quad \sum_{i=0}^{n} \ K^{(n)}_{pq} \, t^p \,.
\end{equation}
The left-hand-side
$
              G_{(\alpha,\beta)}(t)= (1+\alpha t)^{n-q}\; (1+\beta t)^q
$
is thus the {\it generating function} for the 
row entries of the $q^{\mathrm{th}}$ column of $K^{(n)}$.
\begin{proposition}
The Feynman-like sum over paths and the generalized Krawtchouk matrices coincide, i.e.,
$$
\langle\; A  \;\big|\; p\;\rangle \ = \  \ontop n K_{pq} (\alpha,\beta)
$$
\end{proposition}

The proof is analogous to that of Proposition \ref{thm:Feynman}.

\begin{figure}
\vbox{
\bigskip
\hrule
\begin{eqnarray*}
\ontop 0K(\alpha,\beta) 
&=&\left[\begin{array}{c} 1 \end{array}\right]   
\\[3pt]
\ontop 1K(\alpha,\beta) 
&=&\left[\begin{array}{cc} 1 &  1 \cr
                           \alpha & \beta \cr\end{array}\right]    
\\[3pt]
\ontop 2K(\alpha,\beta) 
&=& \left[\begin{array}{ccc}  1 &  1 &  1 \cr
                             2\alpha & \alpha+\beta & 2\beta \cr
                             \alpha^2 & \alpha\beta &  \beta^2 \cr \end{array}\right]
\\[3pt]
\ontop 3K(\alpha,\beta) 
&=&
       \left[\begin{array}{cccc}
                      1 &  1 &  1  &  1 \cr
                      3\alpha &  2\alpha+\beta & \alpha+2\beta  & 3\beta \cr
                      3\alpha^2 &\alpha^2+2\alpha\beta & 2\alpha\beta+\beta^2    &  3\beta^2\cr
                      \alpha^3 & \alpha^2\beta &  \alpha\beta^2  & \beta^3 \cr 
             \end{array}\right]
\\[3pt]
\ontop 4K(\alpha,\beta) 
&=&
       \left[\begin{array}{ccccc} 1 &  1 &  1  &  1  &  1 \cr
                      4\alpha &  3\alpha+\beta &  2\alpha+2\beta  & \alpha+3\beta  & 4\beta \cr
                      6\alpha^2 &  3\alpha^2+3\alpha\beta & \alpha^2+4\alpha\beta+\beta^2  &  3\alpha\beta+3\beta^2  &  6\beta^2 \cr
                      4\alpha^3 & \alpha^3+3\alpha^2\beta &  2\alpha^2\beta + 2\alpha\beta^2  &  3\alpha\beta^2+\beta^3  & 4\beta^3 \cr
                      \alpha^4 & \alpha^3\beta &  \alpha^2\beta^2  & \alpha\beta^3  &  \beta^4 \cr \end{array}\right]
\end{eqnarray*}
\hrule
}
\caption{Krawtchouk matrices with entries in a commutative ring $\mathcal R$}
\label{fig:alphabeta}
\end{figure}

~~

Despite the general character of these definitions, some basic symmetries analogous 
to those for the classical Krawtchouk matrices (\ref{eq:cross}) still hold:

\begin{proposition}
\label{thm:cross}
The cross identities known for the standard Krawtchouk matrices have their analogs for $(\alpha,\beta$) formulation:
$$
\begin{array}{clllll}
             (i)\!  &\alpha \ontop n K_{pq}+ \ontop n K_{p+1,q} \!\!    &= \  \ontop {n+1}K_{p+1,q} 
&\quad (iii)\!  & \ontop n K_{pq} - \ontop n K_{p,q+1}   \!\!  & = \ (\alpha \! - \!\beta) \,\ontop {n-1}K_{p-1,q}
\\
(ii) \!           &\beta \ontop n K_{pq}+ \ontop n K_{p+1,q}   \!\!  & = \ \ontop {n+1}K_{p+1,q+1}
&\quad (iv) \! &\alpha \ontop nK_{p,q+1}-\beta \ontop nK_{pq} \!\!  &= \ (\alpha\! - \! \beta) \, \ontop {n-1}K_{pq}
\end{array}
$$  
\end{proposition}
\noindent
{\bf Proof:} 
Start with 
$\ontop{n+1}K_{iq} \; t^i    =      (1+\alpha t)^{n-q+1}(1+\beta t )^{q}  
= (1+\alpha t) \; (1+\alpha t)^{n-q}(1+\beta t )^{q} 
=
\sum_i
\ontop{n}{K}_{iq} \; t^i  + \sum_i  \alpha \ontop{n}K_{iq} \; t^{i+1}$.
Next, compare coefficients at $t^{p+1}$.
Eq. (ii) resolves the same way.  Eq. (iii) and (iv) are inversions of (i) and (ii).
\QED
\\

In any square of four adjacent entries in Krawtchouk matrix, i.e. a submatrix 
$$
M=\left[\begin{array}{cc} x &  y \\
                           z & t\cr\end{array}\right]    \,,
$$
the following identity holds: $\beta x+ z = \alpha y +t$.
This may be expressed also in a matrix form:
{\small
$$
\hbox{Tr}\,    \begin{bmatrix} \beta &  1 \\
                                                -\alpha & -1\end{bmatrix}    
\begin{bmatrix} x &  y \\
                                                z & t \end{bmatrix} 
\   = \ 0
$$
}
The case of not-necessarily commutative ring and its combinatorial implications will be presented elsewhere.

\section{Quaternions and related Lie groups and algebras}
\label{sec:quaternions}

In this section we revisit the results  of \cite{FK01} providing a different context, justification and meaning.
In particular, we want to clarify that Krawtchouk matrices are elements of 
the (representation of the) group $SL(2,\mathbb R)$
and as such they interact with the Lie algebra $sl(2,\mathbb R)$ via adjoint action.  
The master equation (\ref{eq:master}) is its manifestation.
We use the split quaternions as the algebra tool to deal with this group and algebra.

\subsection{Split quaternions}

The use of quaternions for describing rotations and orthogonal transformations of the Euclidean space $\mathbb R^3$
is well known.
Somewhat less popular is the use of the split quaternions for describing orthonormal transformations 
of a three-dimensional Minkowski space $R^{1,2}$, that is a 3-dimensional real space 
with the inner product and norm defined
by pseudo-Euclidean (hyperbolic) structure of signature $(+,-,-)$. \cite{jk}

The Albert's generalized Cayley-Dickson construction \cite{Albert}
yields two types of four-dimensional algebras:  the ``regular'' quaternions 
$$
\mathbb H = \{a+b\mathbi i +c\mathbi j+d\mathbi k \;\big|\; a,b,c,d\in \mathbi R\}
$$
and the so called split quaternions:
$$
\mathbb K = \{a+b\mathbi i +c\mathbi F+d\mathbi G\; \big|\; a,b,c,d\in \mathbb R\}
$$
In both cases the pure (imaginary) units anticommute and the multiplication rules are:
$$
\begin{array}{lcl} 
\quad \mathbi i^2 =\mathbi j^2=\mathbi k^2 =-1  &\qquad&\quad  \mathbi i^2=-1, \quad \mathbi F^2=\mathbi G^2  =1\\ [7pt]
\begin{array}{rcrcr}
\mathbi {i\/j} &=&\mathbi k &=&-\mathbi {j\/i}  \\
\mathbi {j\/k} &=&-\mathbi i &=&-\mathbi {k\/j}   \\
\mathbi {k\/i} &=&\mathbi j &=&-\mathbi {i\/k}  
\end{array}
&&
\begin{array}{lcrcr}
\ \mathbi {i\/F} &=& \mathbi G &=&- \mathbi{F\/i}  \\
\mathbi {F\/G} &=& -\mathbi i &=& -\mathbi {G\/F}  \\
\mathbi {G\/i} &=& \mathbi F &=& -\mathbi {i\/G}  \\
\end{array}
\end{array}
$$
\begin{center}
\pgfmathsetmacro{\fish}{sqrt(3)}
\begin{tikzpicture}[scale=1.4] 
\coordinate [label=right: { }] (a) at (0.5,\fish/2);  
\coordinate [label=left: { }] (b) at (-0.5,\fish/2);  
\coordinate [label=below: { }] (c) at (0,0);
\path 
(c)  node (nodec) [shape=circle,draw] {\small$\boldsymbol i$}
(a)  node (nodea) [shape=circle,draw] {\small$\boldsymbol k$}
(b)  node (nodeb) [shape=circle,draw] {\small$\boldsymbol j$};
\draw[triangle 60-] (nodea) to [out=120+10,in=60-10]  (nodeb);
\draw[triangle 60-] (nodeb) to [out=240+10,in=180-10] (nodec);
\draw[triangle 60-] (nodec) to [out=0+10,in=-60-10]   (nodea);
   \end{tikzpicture}
\qquad\qquad\qquad
\begin{tikzpicture}[scale=1.4] 
\coordinate [label=below: { }] (c) at (0,0);
\coordinate [label=left: { }] (b) at (-0.5,\fish/2);  
\coordinate [label=right: { }] (a) at (0.5,\fish/2);  
\path 
(c)  node (nodec) [shape=circle,draw] {\small$\mathbi i$}
(b)  node (nodeb) [shape=circle,draw] {\small$\mathbi  F$}
(a)  node (nodea) [shape=circle,draw] {\small$\mathbi  G$};
\draw[triangle 60-] (nodea) to [out=120+10,in=60-10] node [below] {$-$}  (nodeb);
\draw[triangle 60-] (nodeb) to [out=240+10,in=180-10] (nodec);
\draw[triangle 60-] (nodec) to [out=0+10,in=-60-10]   (nodea);
   \end{tikzpicture}
\end{center}
\noindent
{\bf Remark:}  It would be natural to use symbols $\mathbi J$ and $\mathbi K$ for the imaginary units 
in $\mathbb K$, but this could lead to confusion since $K$ denotes Krawtchouk matrices.
Hence we will use $\mathbi F \equiv \mathbi J$ and $\mathbi G \equiv \mathbi K$.
\\


Here are the essential definitions and easy to prove properties:

\begin{table}[H]
\centering
\begin{tabular}{ll|l}
&\bf quaternions $\mathbb H$
&\bf split quaternions $\mathbb K$\\[7pt]
quaternion:  
&$q = a+b\mathbi i +c\mathbi j+d\mathbi k $
&$q = a+b\mathbi i +c\mathbi F+d\mathbi G $
\\[3pt]

conjugation:\quad~
&$\bar q = a-b\mathbi i -c\mathbi j-d\mathbi k 
$&$\bar q = a-b\mathbi i -c\mathbi F-d\mathbi G$
\\[3pt]
norm:
&$\|q\|^2 = q\bar q $
&$\|q\|^2 = q\bar q $
\\
&$\phantom{\|q\|^2} = a^2+b^2 + c^2 + d^2$
&$\phantom{\|q\|^2} = a^2+b^2 - c^2 -d^2$
\\[3pt]
inverse:
&$q^{-1} = \bar q/\|q\|^2$
&$q^{-1} = \bar q/\|q\|^2$
\\[3pt]
pure part:  
&$\Pu q = b\mathbi i +c\mathbi j+d\mathbi k $
&$\Pu q = b\mathbi i +c\mathbi F+d\mathbi G $
\\[3pt]
factor conj'tion:  
&$\overline{pq}= \bar q\, \bar p$
&$\overline{pq}= \bar q\, \bar p$
\\[3pt]
factor norm:  
&$\|pq\|^2 = \|p\|^2 \|q\|^2$
&$\|pq\|^2 = \|p\|^2 \|q\|^2$
\end{tabular}
\end{table}

\noindent
(Symbol $\mathbb F$ will refer to both $\mathbb H$ and $\mathbb K$.) 
Connection of these algebras with geometry is as follows.
The pure part of the quaternions represents a space: 
Euclidean and Minkowski, respectively.
$$
\begin{array}{ccc}
\hbox{space}\ \mathbb R^3 \ \longrightarrow \ \Pu\mathbb H
&\qquad
&\hbox{space}\ \mathbb R^{1,2} \ \longrightarrow \ \Pu\mathbb K\\
\mathbf v = (x,y,z) \mapsto x\mathbi i + y\mathbi j+z\mathbi k 
&\qquad
&\mathbf v = (t,x,y) \mapsto t\mathbi i + x\mathbi F+y\mathbi G \\
\|\mathbf v\|^2 = x^2+y^2+z^2 
&\qquad
&\|\mathbf v\| = t^2-x^2-y^2 
\end{array}
$$
Special orthogonal transformation of space may be performed with unit quaternions:
$$ 
SO(3) \longleftarrow \ \mathbb H_1   \equiv \{q\in\mathbb H \;\big|\; \|q\|^2=1\}
\qquad 
 SO(1,2) \longleftarrow \ \mathbb K_1   \equiv \{q\in\mathbb K \;\big|\; \|q\|^2=1\} 
$$
In both cases, the action is defined:
\begin{equation}
\label{eq:qvq}
\mathbf v\quad\longrightarrow\quad q\mathbf v q^{-1}
\end{equation}
A reflection in a plane perpendicular to vector $q\in \Pu \mathbb F$ is executed by a map:
\begin{equation}
\label{eq:mqvq}
\mathbf v\quad\longrightarrow\quad -q\mathbf v q^{-1}
\end{equation}
(Clearly, a composition of two reflection results in rotation).    
\\[-7pt]

There are a number of coincidences that happen here due to the dimensions of the objects.
Each of the two algebras contains:
\begin{enumerate}
\item
The scalar line $\mathbb R$ for the values of the inner product \\[-22pt]
\item
A copy of the corresponding space as a subspace, $\Pu\mathbb F$. \\[-22pt]
\item
A (double cover) of the orthogonal group acting on the space \\[-22pt]
\item
A copy of the Lie algebra of the group as the $\Pu\mathbb F$,
with the product defined $[a,b] = \frac{1}{2} \, (ab-ba) = \hbox{Re}\, (ab)$.
\end{enumerate}

The Lie algebra commutation rules for both cases are defined for the basis elements as 
$[a,b] = \frac{1}{2}(ab-ba) = ab$.  Thus:
\begin{equation}
\label{eq:commutation}
\begin{array}{lcl} 
\begin{array}{ccc}
\left[\mathbi i, \, \mathbi j\right]  &=&  \mathbi k  \\
\left[\mathbi j,\, \mathbi k\right]  &=&  \mathbi i    \\
\left[\mathbi k,\, \mathbi i\right]  &=&  \mathbi j    
\end{array}
&\qquad&
\begin{array}{ccr}
\left[\mathbi i, \, \mathbi F\right]  &=&  \mathbi G  \\
\left[\mathbi F,\, \mathbi G\right]  &=&  -\mathbi i    \\
\left[\mathbi G,\, \mathbi i\right]  &=&  \mathbi F    
\end{array}
\end{array}
\end{equation}

\noindent
{\bf Remark:} 
Quite interestingly,   
both the space and the Lie algebra are modeled by the same subspace $\Pu \mathbb F$.
The algebraic representation of transformations (\ref{eq:qvq}) and (\ref{eq:mqvq})
is just another view of the adjoint action of group on the corresponding Lie algebra: 
$$
m\to \hbox{Ad}_gm\     \equiv \ gmg^{-1}.
$$

\begin{figure}[h]
\centering
\includegraphics[scale=.7]{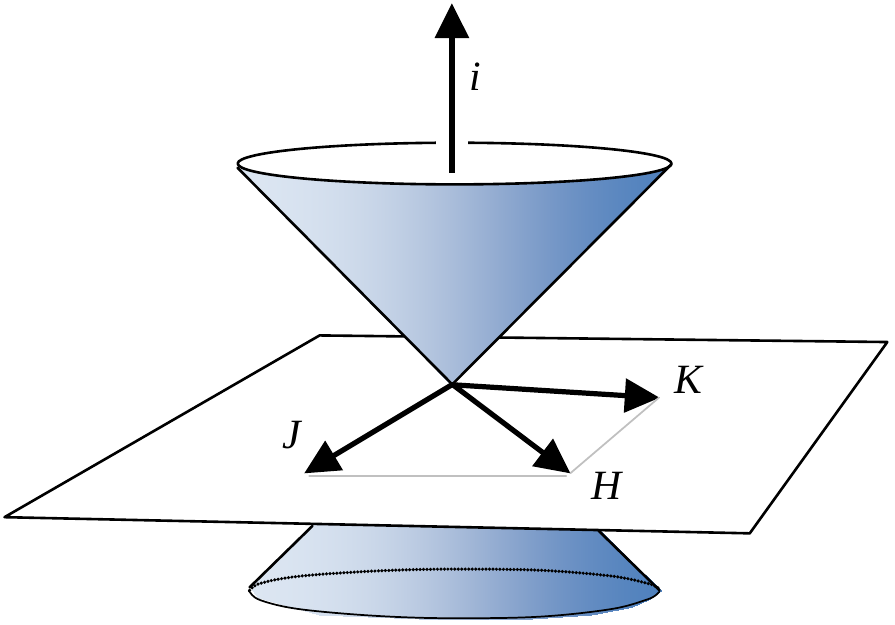}
\caption{Minkowski space $\Pu \mathbb K$}
\label{fig:Minkowski}
\end{figure}

Both quaternion algebras have matrix representations. 
Here is one for the regular quaternions:
$$
\mathbi 1 \ \leftrightarrow \ \begin{bmatrix} 1&0 \\ 0 & 1\end{bmatrix} 
\qquad
\mathbi i \ \leftrightarrow \ \begin{bmatrix} 0&1 \\ -1 & 0\end{bmatrix} 
\qquad
\mathbi j \ \leftrightarrow \ \begin{bmatrix} 0&i \\ i & 0\end{bmatrix} 
\qquad
\mathbi k \ \leftrightarrow \ \begin{bmatrix} i&0 \\ 0 & -i\end{bmatrix} \,,
$$
and here is one for the split quaternions:
$$
\mathbi 1 \ \leftrightarrow \ \begin{bmatrix} 1&0 \\ 0 & 1\end{bmatrix} 
\qquad
\mathbi i \ \leftrightarrow \ \begin{bmatrix} 0&1 \\ -1 & 0\end{bmatrix} 
\qquad
\mathbi F\ \leftrightarrow \  \begin{bmatrix} 0&1 \\ 1 & 0\end{bmatrix} 
\qquad
\mathbi G\ \leftrightarrow \  \begin{bmatrix} 1&0 \\ 0 & -1\end{bmatrix} 
$$
A simple inspection of the matrices reveals that, 
in the first case, $\mathbb H$, the representation of the group action \ref{eq:qvq}
is equivalent to $SU\!(2)$ acting on traceless 2-by-2 skew-Hermitian matrices.
In the second case, the group $\mathbb K_1$ is represented by 
$SL(2,\mathbb R)$ acting on traceless matrices (see \cite{jk}).
\\
\\
{\bf Remark:} 
Although the groups $SO(3)$ and $SU(2)$  are not isomorphic, they have the same Lie algebra $su(2)$.
Similarly, the groups  $SO(1,2)$ and $SL(2,\mathbb R)$ are not isomorphic, but share the same Lie algebra 
$sl(2,\mathbb R)$.  These two Lie algebras become isomorphic under complexification.
They are two real forms of the Lie algebra $sl(2,\mathbb C)$.
Also, we we should not forget that the groups $SU(2)$ and $SL(2,\mathbb R)$  have their own fundamental action
on corresponding two-dimensional spinor spaces, $\mathbb C^2$ and $\mathbb R^2$, respectively.
These spinor actions are essential for quantum computing and for building Krawtchouk matrices.
\\

All these facts and relations are summarized in the diagram in Figure \ref{fig:diagram}.
The action of group $G$ is on a set $X$ is denoted 
by a ``fraction'': $G \gaction{\ \ }{\ } X$.

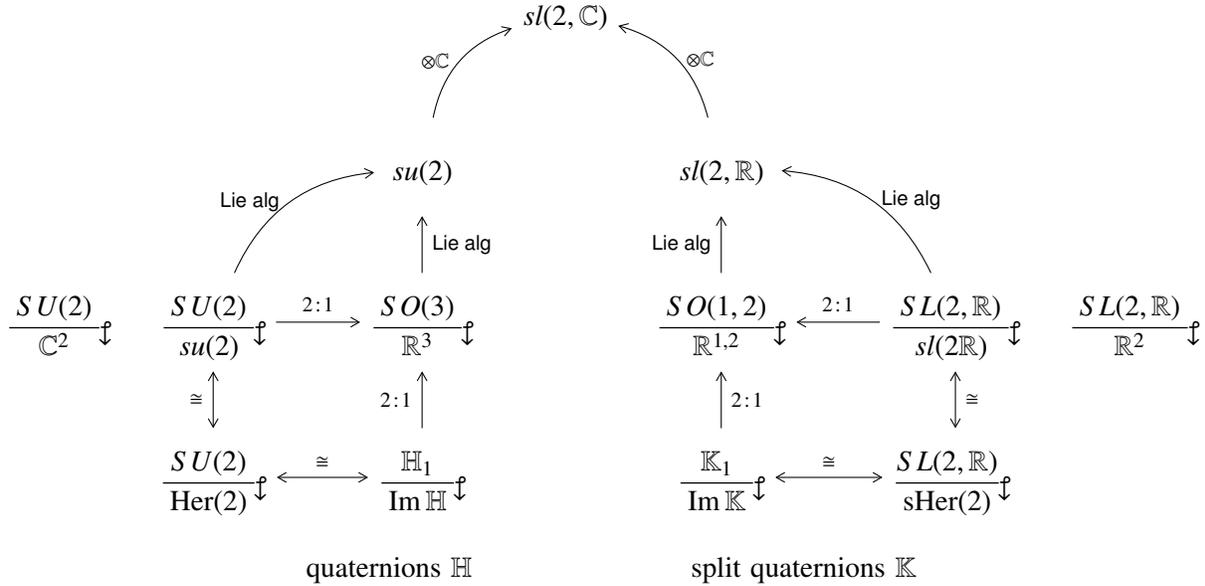
\begin{figure}
\begin{adjustwidth}{-1.5cm}{-2cm}
\centering

$$
\begin{tikzpicture}[baseline=-0.8ex] 
    \matrix (m) [ matrix of math nodes,
                         row sep=2em,
                         column sep=1em,
                         text height=4ex, text depth=2ex] 
   {
  &&&&   sl(2,\mathbb C)  &&&
\\
  &&& \ su(2) \  && \ sl(2,\mathbb R) \ 
\\
\dgaction{~SU(2)~}{\mathbb C^2}  & \dgaction{~SU(2)~}{su(2)}   &~ & \dgaction{~SO(3)~}{\mathbb R^3}  
&  ~            & 
    \dgaction{~SO(1,2)~}{\mathbb R^{1,2}} & ~ & \dgaction{~SL(2,\mathbb R)~}{sl(2\mathbb R)}  
                                                                                               & \dgaction{~SL(2,\mathbb R)~}{\mathbb R^2}  \ 
\\
 & \dgaction{~SU(2)~}{~\hbox{Her(2)~}}   & ~  & \dgaction{~\mathbb H_1~}{~\rIm \mathbb H~}  
&~ &     
     \dgaction{ \mathbb K_1}{~\rIm \mathbb K~}    &  ~  &   \dgaction{ SL(2, \mathbb R)}{~\hbox{sHer}(2)~}  \ 
\\[-20pt]
   ~ &~ & ~  &  ~            & ~ & ~ & ~ & ~ & ~ & ~
\\
     };
    \path[-angle 60]
        (m-2-4) edge [bend left] node[left] {\smalll $\otimes\mathbb C$} (m-1-5)
        (m-2-6) edge [bend right] node[right] {\smalll $\otimes \mathbb C$} (m-1-5)
        (m-3-2) edge [bend left] node[left] {\sf\smalll  Lie alg} (m-2-4)
        (m-3-4) edge node[right] {\sf\smalll Lie alg } (m-2-4)
        (m-3-6) edge node[left]  {\sf\smalll Lie alg} (m-2-6)
        (m-3-8) edge [bend right] node[right] {\sf\smalll Lie alg} (m-2-6)
       (m-3-2) edge node[above] {\smalll $2\!:\!1$} (m-3-4)
       (m-3-8) edge node[above] {\smalll $2\!:\!1$} (m-3-6)
        (m-4-4) edge node[left]  {\sf\smalll $2\!:\!1$} (m-3-4)
        (m-4-6) edge node[right]  {\sf\smalll $2\!:\!1$} (m-3-6);    
    \path[angle 60-angle 60]     
        (m-4-8) edge node[above] {\smalll $\cong$} (m-4-6)
        (m-4-2) edge node[above] {\smalll $\cong$} (m-4-4)
        (m-4-2) edge node[left]  {\sf\smalll $\cong$} (m-3-2)
        (m-4-8) edge node[right]  {\sf\smalll $\cong$} (m-3-8);    
\node (outer) [fit=(m-5-2) (m-5-4) (m-5-5)] {quaternions $\mathbb H$};
\node (outer) [fit=(m-5-5) (m-5-6) (m-5-8)] {\qquad\quad split quaternions $\mathbb K$};
\end{tikzpicture}   
$$
\vspace{-1cm}
\caption{Quaternions, rotations, boosts and all that} 
\label{fig:diagram}
\end{adjustwidth}
\end{figure}

\newpage

\subsection{Krawtchouk matrices from split quaternions}

Now we turn to Krawtchouk matrices.
First note that the fundamental Hadamard matrix may be viewed as 
a space-like element of $\Pu \mathbb K$ (see Figure \ref{fig:Minkowski}):
$$
H=F+G = \left[\begin{array}{rr} 1 &  1 \cr
                       1 & -1 \cr\end{array}\right] \,.
$$
It coincides with the second Krawtchouk matrix.
The simplest form of the master equation (\ref{eq:master}) defining Krawtchouk matrices, namely
\begin{equation}
\label{eq:master-e}
\left[\begin{array}{cc}0&1\cr 1&0\cr\end{array}\right] \
\left[\begin{array}{rr} 1 &  1 \cr
                       1 & -1 \cr\end{array}\right] \ = \ 
\left[\begin{array}{rr} 1 &  1 \cr
                       1 & -1 \cr\end{array}\right] 
\left[\begin{array}{rr}1&0\cr  0&-1\cr\end{array}\right]\,,
\end{equation}
translates to quaternionic equation
\begin{equation}
\label{eq:JHHK}
\mathbi{ FH}  \ = \ \mathbi{HG}
\end{equation}
which, geometrically, is self-evident (see Figure \ref{fig:Minkowski}), and, algebraically,  follows trivially:
$\mathbi{FH} = \mathbi F(\mathbi F+\mathbi G)=(\mathbi F+\mathbi G)\mathbi G =\mathbi H\mathbi G$
(both sides are equal to $1-\mathbi i$.)  
This elementary master equation is the starting point for constructing the higher order versions 
by taking tensor powers of both sides.
But first one needs to recognize the {\it nature} of this equation.
Let us write it as follows:
$$
\mathbi{ F}  \ = \ \mathbi{HGH}^{-1}
$$
Since all terms involved are the elements of $\Pu \mathbb K$, we may interpret the equation in two ways:
(i) as the product of the group elements, namely as the adjoint action of the group on itself,
or (ii) as the adjoint action of the group on the Lie algebra!
(The fact that $H$ is not normalized does not matter since the inverse of $H$ counters any scalar factor in $H$).
\begin{equation}
\label{eq:two}
\begin{array}{ccc}
(i)\ \dgaction{SL(2,\mathbb R)}{SL(2,\mathbb R)}  &\qquad &(ii) \ \dgaction{SL(2,\mathbb R)}{sl(2,\mathbb R)}   \\[11pt]
\hbox{Ad}_{H} \mathbi G = \mathbi F    && \hbox{ad}_{H} {\mathbi G} = {\mathbi F}  \\ [7pt]
 H, \mathbi F, \mathbi G \in SL(2,\mathbb R) &&  H \in SL(2,\mathbb R),\ \mathbi {F,G} \in sl(2,\mathbb R)
\end{array}
\end{equation}
Seeking the higher-order tensor representations,  we need to treat these two interpretations differently.
Recall that the tensor powers of group elements and tensor powers of Lie algebra elements are formed differently:
$$
\begin{array}{rclc} 
\hbox{group:}    &\quad&  g \longrightarrow  g\otimes g \\
\hbox{algebra:}   &&         m \longrightarrow   m\otimes I + I\otimes m  
\end{array}
$$
This way the corresponding algebraic structures are preserved: for any two elements $g,h\in G$  and $m,n\in Lie\ G$ we have
$$
\begin{array}{rcrcc} 
\hbox{group:} &\quad&(g\otimes g)\;(h\otimes h) \ &=& \ (gh)\otimes (gh)   \\
\hbox{algebra:}   &&\left[ m\!\otimes\! I   + I\!\otimes\! m, \ n\!\otimes\! I + I\!\otimes\! n\right]  &=&  [m,n] \otimes I + I\otimes [m,n]
\end{array}
$$
The same goes for the higher tensor powers.
\\

Here is a geometric justification. For simplicity, assume that our objects are matrix representations.
Geometrically, an element $A$ of the Lie algebra is a vector tangent to the group at identity $I$.
We may design a curve, the velocity of which at $I$ is $A$:
$$
A_t := I+tA, \qquad \frac{d}{dt}\Bigg|_{t=0} A_t = A
$$
Although the curve is a straight line that does not lie in the group, it does not matter: 
it lies in the vector space of the matrices and we need only its behavior at $I$, 
where it happens to be tangent to the group. 
Now, the $n$-th tensor power will result as the velocity of the curve at $t=0$ of the tensored curve, i.e.: 
\begin{equation}
\label{eq:Liebox}
\begin{aligned}
A^{\boxtimes  n} &= \frac{d}{dt}\Bigg|_{t=0} \, (I+tA)^{\otimes n} \\[7pt]
                          & = \ A\!\otimes\! I\!\otimes\! ... \!\otimes\! I   \ + \
                                I\!\otimes\! A\!\otimes\! ... \!\otimes\! I \ + \ ... \ + \
                                I\!\otimes\! ... \!\otimes\! I\! \otimes\! A
\end{aligned}
\end{equation}
We use the boxed version $\boxtimes$ instead of $\otimes$ to indicate this form of tensoring.
\\

\noindent
{\bf Symmetric powers.}
Krawtchouk matrices emerge when one applies the symmetric tensor power 
to the elements of the elementary master equation (\ref{eq:master-e}).
The two identifications bring two different results:

\begin{proposition}
The symmetric tensor power of (i) and (ii) are respectively:
\begin{equation}
\label{eq:twoo}
\begin{array}{ccc}
(i)\ \dgaction{SL(2,\mathbb R)}{sl(2,\mathbb R)} 
&\qquad &
(ii) \ \dgaction{SL(2,\mathbb R)}{SL(2,\mathbb R)}    \\[11pt]
        \Downarrow    && \Downarrow  \\ [7pt]
 MK \ = \ K\Lambda
&&
K_{pq} = (-1)^q K_{n-p,q}, \\
&&K_{pq}= (-1)^p\, K_{p,n-q}\\
\hbox{(Master equation)} &&\hbox{(symmetries)}

\end{array}
\end{equation}
\end{proposition}

\begin{proof}
Let $V$ be a two-dimensional space with the basis
$$  
V = \rmspan\{e_0,\, e_1\}
$$
The $n$-th {\bf symmetric} tensor power of this space acquires a basis from $V$, namely
$$
V^{\odot n} = \rmspan \{ e_0\odot\cdots\odot e_0,\;     e_1\odot e_0\odot\cdots\odot e_0,\;  ...,  \;  e_1\odot\cdots\odot e_1,\,\}
$$
(Recall that $v\odot w = w\odot v$.)
The dimension of this space is $\dim V^{\odot n} = n+1$.
Any endomorphism $T\in \hbox{End}\, V$ induces an endomorphism in the symmetric tensor power space.
Depending on the interpretation of $T$ (group versus algebra), the two cases are denoted 
$$
T^{\odot n} \quad \hbox{if}\quad  T \in SL(2,\mathbb R)
\qquad 
T^{\boxdot} \quad \hbox{if}\quad  T \in sl(2,\mathbb R)
$$   
Presentation and calculations become much simplified if one applies the following trick. 
Replace the  basis vectors by commutative variables,  $e_0\to x$ and $e_1\to y$.
Thus the tensor space is represented by the homogeneous polynomials of order $n$:
$$
V^{\odot n} \ \cong \ \rmspan\{ x^n,\, x^{n-1}y, \, x^{x-2}y^2,\, ...,\, y^n \} \equiv \hbox{Poly}_n(x,y)
$$
The right hand side represents actually the symmetric tensor power of the {\it dual}  space $V^*$.
Hence, in the further construction of the endomorphisms, we shall use the transposition of $A$ rather.
Denote the elements of the basis:
$$
\mathbf e_i = x^{n-i}y^i   \qquad i=0,,...,n
$$
Define a family of maps $\varepsilon_i  : V\to \hbox{Poly}_n(x,y)$
\begin{equation}
\label{eq:epsilon}
\varepsilon_i \, \begin{bmatrix}x\\y\end{bmatrix}  = x^{n-i}y^i \equiv \mathbf e_i
\end{equation}
Any endomorphisms in  $A:V\to V$ induces two types of endomorphism in $V^{\odot n}$ 
according to these {\bf symmetric tensor power rules}:
\begin{equation}
\label{eq:rules}
\begin{array}{rclcc} 
\hbox{for {\it A} as a group element:}    &\quad&  A^{\odot n} \circ \varepsilon_i \ = \  \varepsilon_i \circ A^T       &\qquad(i)\\[3pt]
\hbox{for {\it A} as an algebra element:}   &&    A^{\boxdot n} \circ \varepsilon_i \ =  \  \frac{d}{dt}\,\varepsilon_i \circ (I+A^Tt)   &\qquad(ii)
\end{array}
\end{equation}
where the derivative is to be evaluated at $t=0$, and $A^T$ is the transposition of $A$.
\\

With these tools at hand we can turn to the main task.
\\

\noindent
{\bf Part (i):}
First, calculate the group action on the algebra.  Thus $F,G\in sl(2,\mathbb R)$ and $H\in SL(2,\mathbb R)$
(the scaling is irrelevant).
Applying rules (\ref{eq:rules}), we get:
$$
\begin{array}{cccll}
F &\to& F^{\boxdot n} &= \diag\left(\begin{array}{ccccccccccccccc}
                                                        \z  &\z 1&\z &\z 2&\z &\z 3&\z &\z ...&\z &\z n\\ 
                                                       0& &\z 0 &\z &\z 0&\z &\z  ...&\z &\z 0&\z &\!\!\! 0 \\
                                                      \z   &\z n&\z &\z n\!-\! 1&\z &\z n\!-\! 2&\z &\z ...&\z &\z 1\end{array}\right) &= \ontop n M \\[17pt]
G &\to& G^{\boxdot n} &= \diag\left(\begin{array}{ccccccccccccccc}
                                                            n& n-2 & n\!-\! 4&...&-n\!+\! 2&-n \end{array}\right)  &= \ontop n \Lambda \\[14pt]
H &\to& H^{\odot n} &=\ontop nK \qquad\hbox{(Krawtchouk matrix)}
\end{array}
$$
This yields exactly the master equation (\ref{eq:master}) with $M$=Kac matrix and $\Lambda$=eigenvalue matrix.
Here are the calculations. Note that $H^T=H$ and use rule (\ref{eq:rules} ({\it i}\/)):
$$
\begin{aligned}
H^{\odot n}  \mathbf e_q   
&= \left[\begin{smallmatrix}1&\phantom{-}1\\1&-1\end{smallmatrix}\right]^{\odot n}  \varepsilon_q\left(\left[\begin{smallmatrix}x\\y\end{smallmatrix}\right]\right)  
= \varepsilon_p  \left(\left[\begin{smallmatrix}1&\phantom{-}1\\1&-1\end{smallmatrix}\right]\,\left[\begin{smallmatrix}x\\y\end{smallmatrix}\right]\right)
= \varepsilon_p  \left(\left[\begin{smallmatrix}x+y\\x-y\end{smallmatrix}\right]\right)  \\[7pt]
&  =   (x+y)^{n-q}(x-y)^q  
= \sum_p K_{pq} x^{n-p}y^p  = \sum_p K_{pq} \mathbf e_p 
\end{aligned}
$$
Thus, indeed, $H^{\odot n} = \ontop n K$.  
Now, as a Lie algebra element, transformation $F$ requires the second rule of (\ref{eq:rules}): 
$$
\begin{aligned} 
F^{\boxdot n} \mathbf e_q  \ &= \ \frac{d}{dt}\Bigg|_{t=0} \varepsilon_q \left( (I+tF) {\scriptstyle \begin{bmatrix}x\\y\end{bmatrix}}\right) 
                                      \ \ = \ \ \frac{d}{dt}\Bigg|_{t=0} \ \varepsilon_q {\scriptstyle\begin{bmatrix}1&t\\t&1\end{bmatrix} \begin{bmatrix}x\\y\end{bmatrix}} \\[7pt]
                                        &= \ \frac{d}{dt}\Bigg|_{t=0}  \varepsilon_q {\scriptstyle\begin{bmatrix}x+ty\\tx+y\end{bmatrix}}   
                                      \   \ = \ \ \frac{d}{dt}\Bigg|_{t=0}  (x+ty)^{n-q}(tx+y)^q \\[7pt]
                                      &= \  (n-q)x^{n-q-1}y^{q+1}  +  q x^{q+1}y^{q-1}  \\[7pt]
                                      &= \  (n-q)\mathbf e_{q+1} + q \mathbf e_{q-1}  
\end{aligned}
$$
(we again used the fact that $F^T = F$). This recovers matrix $M$.
Matrix  $G^{\boxdot n}$ is calculated similarly.
The master equation is recovered.
\\

\noindent
{\bf Part (ii):}  As group elements, the matrices $F$, $G$ and $H$  are tensored as follows:
$$
\begin{array}{cccl}
F &\to& F^{\odot n} &= \hbox{skewdiag}\left(\begin{array}{ccccccccccccccc}
                                                            1& 1 & 1&...&1&1 \end{array}\right)\\[7pt]
G &\to& G^{\odot n} &= \hbox{diag}\left(\begin{array}{ccccccccccccccc}
                                                            1& -1 & -1&1&...&{-1}^n \end{array}\right)\\[7pt]
H &\to& H^{\odot n} &= \ontop nK
\end{array}
$$ 
These will lead to symmetries of $K$. 
Consider two versions of the quaternionic master equation:
$$
(i)\qquad FK=KG \qquad (ii)\quad GK = KF
$$
They lead to 
$$
(i)\qquad F^{\odot n} \ontop nK= \ontop nK G^{\odot  n}  \qquad 
(ii) \quad G^{\odot n} \ontop n K = \ontop nK F^{\odot n} 
$$
The first says that inverting the order of rows changes sign at every other column.
The second says that inverting the order of columns changes sign at every other row column. 
$$
(i)\qquad K_{p,q} = (-1)^q K_{n-p,q}  \qquad 
(ii) \quad K_{p,q} = (-1)^p K_{p,n-q} 
$$
Here are examples for $n=3$.
Equation $F^{\odot n} \ontop n K=\ontop n KG^{\odot n}$ becomes (after adjustment):

$$
\left[\begin{array}{rrrr}   
& & &1  \\
& &1 & \\
&1 & & \\
1& & &   \end{array}\right]
\ 
\left[\begin{array}{rrrr}   
1&1 &1 &1  \\
3&1 &-1 &-3 \\
3&-1 &-1 &3 \\
1&-1 &1 &-1  \end{array}\right] 
\
\left[\begin{array}{rrrr}   
1& & &  \\
&-1 & & \\
& &1 & \\
& & &-1  \end{array}\right]
\ = \
\left[\begin{array}{rrrr}   
1&1 &1 &1  \\
3&1 &-1 &-3 \\
3&-1 &-1 &3 \\
1&-1 &1 &-1    \end{array}\right] 
$$
And similarly for Equation $G^{\odot n} \ontop n K=\ontop n KF^{\odot n}$ gives:
$$
\left[\begin{array}{rrrr}   
1& & &  \\
&-1 & & \\
& &1 & \\
& & &-1  \end{array}\right]
\ 
\left[\begin{array}{rrrr}   
1&1 &1 &1  \\
3&1 &-1 &-3 \\
3&-1 &-1 &3 \\
1&-1 &1 &-1  \end{array}\right] 
\
\left[\begin{array}{rrrr}   
& & &1  \\
& &1 & \\
&1 & & \\
1& & &   \end{array}\right]
\ = \
\left[\begin{array}{rrrr}   
1&1 &1 &1  \\
3&1 &-1 &-3 \\
3&-1 &-1 &3 \\
1&-1 &1 &-1    \end{array}\right] 
$$

\end{proof}

\noindent
{\bf Remark on the regular tensor powers: }
Interpreting (\ref{eq:master-e}) as the adjoint action of the group $SL(2,\mathbb R)$ on 
on its Lie algebra, its $n$-th tensor power becomes
$$
G^{\boxtimes n} H^{\otimes n} = H^{\otimes n} F^{\boxtimes n}
$$
where $\otimes$ denotes the regular tensor product and $\boxtimes$ as in (\ref{eq:Liebox}).
The term $H^{\otimes n}$ may be viewed as  $2^n\times 2^n$ Hadamard-Sylvester matrix,
typically represented by the Kronecker products of $H$.
It has quantum interpretation:
One uses $H^{\otimes n}$  to prepare the state of superposition of all numbers $\{0,1,...,n\}$ for further quantum computing.
Suppose one of the states has its phase changed by a $e^{\pi i}$.  
Suppose we do not know which qubit experienced this change and we have a superposition of individual changes.
Then, to undo it, it suffices to apply a similar superposition of the flips $F$.

\subsection{Skew-diagonalization and eigenvectors of Krawtchouk matrices}

Krawtchouk matrix $\ontop n K$ defines {\bf Krawtchouk transform} understood as a discrete map $\mathbb Z^{n+1}\to\mathbb Z^{n+1}$.
Here is an interesting property that may be important for applications: 
\begin{theorem}
\label{thm:binomialtransf}
The Krawtchouk transform of the binomial distribution is a (complementary) binomial distribution.
Namely, let $\mathbf b^{(k)}$ be defined as a vector with the first top entries being binomial coefficients: 
$\mathbf b^{(k)}_i={k\choose i}$ (and 0's for $i>k$).
Then  
\begin{equation}
\label{eq:Kb2b}
K\mathbf b^{(k)} = 2^k \cdot \mathbf b^{(n-k)}
\end{equation}
where $K$ is the $n$-th order Krawtchouk matrix.  This can be expressed collectively in the matrix form
\begin{equation}
\label{eq:KBBD}
     KB = BD
\end{equation}
where $B=[\mathbf b_{(0)}|...|\mathbf b_{(n)}]$, i.e., $B_{ij} =  {j\choose i}$  \ (and $B_{ij} = 0$ if $i>j$)
and $D$ is a skew-diagonal exponential matrix defined $D_{i,n-i} = 2^i$ (and 0 for other entries).
\end{theorem}

\noindent
{\bf Proof:} Start with the easy-to-verify identity: 
\begin{equation}
\label{eq:easy}
   \left[\begin{array}{rr} 1 &  1 \cr
                                    1 & -1  \end{array}\right] \;
   \left[\begin{array}{rrrr} 1 &  1  \cr
                                      0 &  1  \cr \end{array}\right] \;
\ = \ 
   \left[\begin{array}{rrrr} 1 &  1 \cr
                                      0 &  1  \cr \end{array}\right] \;
   \left[\begin{array}{rrrr} 0 &  2  \cr
                                      1 &  0 \cr \end{array}\right] \,.
\end{equation}
Then exponentiate it to the symmetric $n$-th tensor power using the rules (\ref{eq:rules}) 
applied to the above terms understood as group elements.
\QED

~\\
{\bf Example:}  For $n=3$ we have: 
$$
   \left[\begin{array}{rrrr} 1 &  1 &  1  &  1 \cr
                                      3 &  1 & -1  &  -3 \cr
                                      3 & -1 & -1  &  3 \cr
                                      1 & -1 &  1  & -1 \cr \end{array}\right] \;
   \left[\begin{array}{rrrr} 1 &  1 &  1  &  1 \cr
                                         &  1 &  2  &  3 \cr
                                         &    &  1  &  3 \cr
                                         &    &      &  1 \cr \end{array}\right] \;
\ = \ 
   \left[\begin{array}{rrrr} 1 &  1 &  1  &  1 \cr
                                        &  1 &  2  &  3 \cr
                                        &    &  1  &  3 \cr
                                        &   &   &  1 \cr \end{array}\right] \;
   \left[\begin{array}{rrrr}  &   &   &  8 \cr
                                       &    &  4  &   \cr
                                       &  2&   &   \cr
                                      1 &  &    &   \cr \end{array}\right] \;
$$
(empty entries denote 0's.)
Note that this identity may serve as a definition of the Krawtchouk matrices: 
$$
            K=BDB^{-1}
$$
The inverse matrix on the right side can be calculated simply as $B^{-1} = DBD$ where $D = \hbox{diag}\,(1,-1,1...)$.
Here is an example:
$$
   \left[\begin{array}{rrrr} 1 &  1 &  1  &  1 \cr
                                      3 &  1 & -1  &  -3 \cr
                                      3 & -1 & -1  &  3 \cr
                                      1 & -1 &  1  & -1 \cr \end{array}\right] \;
\ = \ 
   \left[\begin{array}{rrrr} 1 &  1 &  1  &  1 \cr
                                       &  1 &  2  &  3 \cr
                                       &   &  1  &  3 \cr
                                       &   &    &  1 \cr \end{array}\right] \;
   \left[\begin{array}{rrrr}  &   &    &  8 \cr
                                       &   &  4 &     \cr
                                       &  2 &   &     \cr
                                      1 &      &      &  \cr \end{array}\right] \;
   \left[\begin{array}{rrrr} 1 &  -1 &  1  &  -1 \cr
                                         &   1 & - 2 &  3 \cr
                                         &      &  1  &  -3 \cr
                                         &      &      &  1 \cr \end{array}\right] 
$$
Such an identity may be called ``skewsymmetrization of Krawtchouk matrices 
and results directly from a symmetric tensor power of the Equation \ref{eq:easy} written as:
$$
   \left[\begin{array}{rr} 1 &  1 \cr
                                    1 & -1  \end{array}\right] \;
\ = \ 
   \left[\begin{array}{rrrr} 1 &  1 \cr
                                      0 &  1  \cr \end{array}\right] \;
   \left[\begin{array}{rrrr} 0 &  2  \cr
                                      1 &  0 \cr \end{array}\right] \;
  \left[\begin{array}{rrrr} 1 &  -1  \cr
                                      0 &  1  \cr \end{array}\right] \;
$$

Identity (\ref{eq:KBBD}) leads in a very simple and elegant way to the the eigenvectors of the Krawtchouk matrices.

\begin{theorem}
The eigenvectors of $\ontop n K$ are the following linear combinations of the complementary binomial vectors:
\begin{equation}
\label{eq:eigenvectors}
   \mathbf v^{\pm}_{(k)}  =  2^{\frac{n-k}{2}}\, \mathbf b_{(k)} \pm 2^{\frac{k}{2}}\, \mathbf b_{n-k}
\end{equation}
with the eigenvalues $\pm 2^{\frac{n}{2}}$ with the corresponding signs:
$K\mathbf v^{\pm}_{(k)} = \pm 2^{n/2} \,\mathbf v^{\pm}_{(k)}$.
The spectral decomposition of $K$ has the following matrix form:
\begin{equation}
\label{eq:eigenvectorsm}
 K(BX) = (BX) E
\end{equation}
where $B$ is the ``binomial matrix'' as in \ref{eq:KBBD} while 
$E$ is a diagonal matrix of eigenvectors and $X$ is a sum of a diagonal and a skew-diagonal matrix:
$$
X_{ij}=\begin{cases} \ +2^{\frac {n-j}{2}}  & \hbox{if } i=j  \hbox{ and }  j\leq n/2\\ 
                                \ -2^{\frac {n-j}{2}}  & \hbox{if } i=j  \hbox{ and }  j> n/2\\ 
                               \ \phantom{-}2^{\frac {j}{2}}      & \hbox{if } i+j=n \\ 
                               \ \phantom{-}0                            & \hbox{otherwise}
\end{cases}
\qquad
E_{ij}=\begin{cases} \ +\sqrt{n}   & \hbox{if } i=j \hbox{ and }  j\leq n/2\\ 
                               \  -\sqrt{n}   & \hbox{if } i=j \hbox{ and }  j>n/2\\
                               \ \phantom{-} 0                            & \hbox{otherwise}
\end{cases}
$$
Eigenvectors of $K$ are the columns of matrix $BX$.
\end{theorem}

\noindent{\bf Example:} Here is an even-dimensional case for  (\ref{eq:eigenvectorsm}):
$$
   \left[\begin{array}{rrrr} 1 &  1 &  1  &  1 \cr
                                      3 &  1 & -1  &  -3 \cr
                                      3 & -1 & -1  &  3 \cr
                                      1 & -1 &  1  & -1 \cr \end{array}\right] \;
   \left[\begin{array}{rrrr} 1 &  1 &  1  &  1 \cr
                                       &  1 &  2  &  3 \cr
                                       &   &  1  &  3 \cr
                                       &   &    &  1 \cr \end{array}\right] \;  
 \left[\begin{array}{rrrr}  \sqrt{8} &             &    & \sqrt{8} \cr
                                                  & \sqrt{4} & \sqrt{4}  &   \cr
                                                  & \sqrt{2} & -\sqrt{2}  &   \cr
                                       \sqrt{1}&             &    &  -\sqrt{1} \cr \end{array}\right] \; \qquad
$$$$
\qquad\ = \ 
\sqrt{8}
   \left[\begin{array}{rrrr} 1 &  1 &  1  &  1 \cr
                                       &  1 &  2  &  3 \cr
                                       &   &  1  &  3 \cr
                                       &   &    &  1 \cr \end{array}\right] \;
 \left[\begin{array}{rrrr}  \sqrt{8} &             &    & \sqrt{8} \cr
                                                  & \sqrt{4} & \sqrt{4}  &   \cr
                                                  & \sqrt{2} & -\sqrt{2}  &   \cr
                                       \sqrt{1}&             &    &  -\sqrt{1} \cr \end{array}\right] \;
  \left[\begin{array}{rrrr}  1&   &    &   \cr
                                       &  1 &   &     \cr
                                       &   &  -1 &     \cr
                                       &      &      & -1 \cr \end{array}\right] \;
$$
The matrix of eigenvectors is thus:
$$
BX=
   \left[\begin{array}{cccc} \sqrt{8}+1 &  2+\sqrt{2}   &  2-\sqrt{2}  &  2\sqrt{2}-1 \cr
                                       3              &  2+2\sqrt{2} &  2-2\sqrt{2}  &  -3 \cr
                                       3              &   \sqrt{2}      &  -\sqrt{2}  &  -3 \cr
                                       1              &    0              &    &  -1 \cr \end{array}\right] 
$$
Here is an odd-dimensional case for comparison:
$$
   \left[\begin{array}{rrrr} 1 &  1 &  1   \cr
                                      2 &  0 & -2   \cr
                                      1 & -1 &  1   \cr \end{array}\right] \;
   \left[\begin{array}{rrrr} 1 &  1 &  1   \cr
                                         &  1 &  2   \cr
                                         &   &  1   \cr \end{array}\right] \;  
 \left[\begin{array}{rrrr}  \sqrt{4} &             &   \sqrt{4} \cr
                                                  & \sqrt{2} &   \cr
                                       \sqrt{1}&             &  -\sqrt{1} \cr \end{array}\right] \; \ 
= \left[\begin{array}{rrrr} 1 &  1 &  1   \cr
                                      2 &  0 & -2   \cr
                                      1 & -1 &  1   \cr \end{array}\right] \;
   \left[\begin{array}{rrrr} 3 &  \sqrt{2} &  3   \cr
                                       2  & \sqrt{2} &  -2   \cr
                                       1  & 0  &  -1   \cr \end{array}\right] \;  
$$$$
\qquad\ = \ 
\sqrt{4}
   \left[\begin{array}{rrrr} 1 &  1 &  1   \cr
                                         &  1 &  2   \cr
                                         &   &  1   \cr \end{array}\right] \;  
 \left[\begin{array}{rrrr}  \sqrt{4} &             &   -\sqrt{4} \cr
                                                  & \sqrt{2} &   \cr
                                       \sqrt{1}&             &    \sqrt{1} \cr \end{array}\right] \;
  \left[\begin{array}{rrrr}  1&   &       \cr
                                       &  1 &        \cr
                                       &   &  -1     \cr \end{array}\right] \;
$$

\subsection{Representations of $sl(2)$}

As is well known, 
the Lie algebra $sl(2,\mathbb R)$ 
has a unique irreducible representation in every dimension (up to isomorphism).
It may be made concrete by choosing
a space of homogeneous polynomials as the $(n+1)$-dimensional representation space,   
$$
V_n  \ =  \ \hbox{span}\, \{ x^{n-i}y^i \;\big|\; i = 0,...,n\} \,, 
$$
and representing the generators of the algebra by differential operators 
\begin{equation}
\label{eq:LR}
 L = x\frac{\q}{\q y} \qquad\hbox{and}\qquad 
 R = y\frac{\q}{\q x}\,,
\end{equation}
called the lowering and raising operators.
For short, we shall write $\q_x = \frac{\q}{\q x}$.
Their commutator is 
$$
N= [L,R] = x\q_x - y\q_y \,.
$$
They span together the algebra with the commutation relations:
\begin{equation}
\label{eq:LRN}
[L,R]=N\qquad
[N,L] = L\qquad
[N,R] = - R
\end{equation}
Here is a pictorial version of the situation for $n=3$:
\begin{equation}
\begin{tikzpicture}[baseline=-0.8ex]
    \matrix (m) [ matrix of math nodes,
                         row sep=2em,
                         column sep=3.5em,
                         text height=3ex, text depth=3ex] 
   {
 \ x^3 \ & \ x^2y\ & \ xy^2\  &\  y^3\ \\
\small 3&\small 1&\small -1&\small -3 \\
     };
    \path[->]
        (m-1-1) edge [bend left] node[above] {\small $y\q_x$} node[below] {\smalll (3)} (m-1-2)
        (m-1-2) edge [bend left] node[above] {\small $y\q_x$} node[below] {\smalll (2)} (m-1-3)
        (m-1-3) edge [bend left] node[above] {\small $y\q_x$} node[below] {\smalll (1)} (m-1-4)
        (m-1-2) edge [bend left] node[below] {\small $x\q_y$} node[above] {\smalll (1)} (m-1-1)
        (m-1-3) edge [bend left] node[below] {\small $x\q_y$} node[above] {\smalll (2)} (m-1-2)
        (m-1-4) edge [bend left] node[below] {\small $x\q_y$} node[above] {\smalll (3)}(m-1-3);
    \path[dotted,thick,  ->]
        (m-1-1) edge node[left] {\smalll $N=y\q_x\!\!-\!\! x\q_y$}  (m-2-1)
        (m-1-2) edge node[left] {\smalll $N$}  (m-2-2)
        (m-1-3) edge node[left] {\smalll $N$}  (m-2-3)
        (m-1-4) edge node[left] {\smalll $N$}  (m-2-4);
\end{tikzpicture}
\qquad   
\end{equation}

\vspace{-19pt}

\noindent
The numbers in the brackets are the factors acquired by applying $L$ or $R$ to particular monomials.
The monomials are the eigenvectors of the operator 
$N = [L,R]$ 
with the corresponding eigenvalues reported in the bottom line, under dotted arrows.
Note the  same pattern as in the examples of Section \ref{sec:examples}.
In particular, compare it with Figure \ref{fig:spins}, 
as well as with the structure of the matrices $M$ and $\Lambda$ of the master equation.
This shows the intrinsic analogies between the Bernoulli random walk, the Lie algebra $sl(2)$, 
and the Krawtchouk matrices.
\\

In the physics literature, such a construction is called Fock representation.
We want to tie it with our previous discussion.

\begin{proposition}
Symmetric power tensoring is additive, i.e.,
\begin{equation}
\label{eq:linear}
(A+B)^{\boxdot n} = A^{\boxdot n}+B^{\boxdot n}
\end{equation}
\end{proposition}

\begin{proof}
It follows directly from the rules (\ref{eq:rules}).
\end{proof}

~\\

\begin{proposition}
\label{prop:many}
The following are the differential operators corresponding to the standard elements of the 
Lie algebra $sl(2,\mathbb R)$ when interpreted as acting on the spaces of the homogeneous polynomials:
$$
\begin{array}{ccllc}
F =&\begin{bmatrix}0&1\\ 1&0\end{bmatrix}   \ &\longrightarrow \  y\/\q_x + x\/\q_y   &\qquad\hbox{(Kac operator)}\\[14pt]
B= &\begin{bmatrix}0&1\\ -1&0\end{bmatrix}   \ &\longrightarrow \  y\/\q_x - x\/\q_y  &\qquad\hbox{(representation of}\ \mathbi i) \\[14pt]
G= &\begin{bmatrix}1&0\\ 0&-1\end{bmatrix}   \ &\longrightarrow \  x\/\q_x - y\/\q_y  &\qquad\hbox{(number operator)}\\[14pt]
L =&\begin{bmatrix}0&1\\ 0&0\end{bmatrix}   \ &\longrightarrow \  x\/\q_y  &\qquad\hbox{(lowering operator)}\\[14pt]
R =& \begin{bmatrix}0&0\\ 1&0\end{bmatrix}   \ &\longrightarrow \  y\/\q_x  &\qquad\hbox{(raising operator)}
\end{array}
$$
In general:
$$
\begin{bmatrix}\alpha&\beta\\\gamma&\delta\end{bmatrix}   
\ \ \longrightarrow \  \
\alpha x\,\q_x + \beta x\,\q_y +\gamma y\,\q_x + \delta y\,\q_y 
$$
\end{proposition}

\begin{proof}
Consider the case of the raising operator.
Remembering to take the transposition, we have 
$$
\begin{aligned} 
R^{\boxdot n} \mathbf e_q  \ &= \ \frac{d}{dt}\Bigg|_{t=0} \varepsilon_q 
                                                   \left( (I+tR^T) {\scriptstyle \begin{bmatrix}x\\y\end{bmatrix}}\right) 
                                          \ = \ \frac{d}{dt}\Bigg|_{t=0}  \varepsilon_q {\scriptstyle\begin{bmatrix}x+ty\\y\end{bmatrix}} \\[7pt]  
                                      & = \ \ \frac{d}{dt}\Bigg|_{t=0}  (x+ty)^{n-q}(y)^q 
                                         \ = \  (n-q)x^{n-q-1}y^{q+1}   \\[7pt]
\end{aligned}
$$
Now one can observe that this is indeed equivalent to action of the operator $y\q_x$.
But we can get this result directly.
Recall the elementary fact about directional derivative: 
directional derivative at point in the direction of variable $x_i$ with a speed $v$ of some scalar function $f(x_1,....,x_n)$ is  
$$
v\,\q_{x_i} \, f = \frac{d}{dt}\Bigg|_p f(x_1,\,...\,,\,x_i+vt,\, ...\,,\, x_n)
$$
In our case, we simply observe that
$$
\frac{d}{dt}\Bigg|_{t=0} f(x+ty,y) \ = \  y\/\q_x \, f(x,y)
$$
where $f(x,y)$ is short for any of the homogeneous polynomial under the action of $L$:
$$
f(x+ty,y) = \varepsilon_i   \left( (I+tL^T) {\scriptstyle \begin{bmatrix}x\\y\end{bmatrix}}\right) 
$$
The other cases, as well as the general case,  resolve in a similar way.
\end{proof}

The linearity (\ref{eq:linear}) may be observed in the examples of Proposition \ref{prop:many}. 
The fundamental Hadamard matrix, interpreted as the element of the Lie algebra, is
$H= (x+y)\q_x + (x-y)\q_y$.  But, as we have argued, it does not enter the master equation.   
\\

We conclude with some general remarks.
Any triple of linearly independent vectors in $\Pu \mathbb K$ span the algebra $sl(2)$,
consult Figure \ref{fig:Minkowski2}.
Two choices are standard: 
(a) the regular basis reflecting the pseudo-orthogonal structure of $\mathbb R^{1,2}$, and 
(b) one that involves isotropic vectors (``light-like'') in the plane of $(\mathbi F,\mathbi i)$:
$$
(a)\quad (\mathbi F,\mathbi G,\mathbi i)
\qquad 
(b)\quad (\mathbi R= {\scriptstyle\frac{1}{2}}\, (\/ \mathbi F+\mathbi i \/),\  
\mathbi L= {\scriptstyle\frac{1}{2}}\, (\/ \mathbi F-\mathbi i\/),\
\mathbi i)
$$
The second choice corresponds to (\ref{eq:LR}) with identification $\mathbi i = N$.
For completeness, here is the matrix representation of $i$ when power-tensored as an algebra element
for $n=3$:
$$
\ontop 3 B = 
\left[\begin{array}{rrrr} 
         0 & 1 & 0 & 0 \cr
         -3 & 0 & 2 & 0 \cr
         0 & -2 & 0 & 3 \cr
         0 & 0 & -1 & 0 \cr \end{array}\right]\,,
$$
and similarly for other degrees.

\vspace{-12pt}

\begin{figure}[h]
\centering
\includegraphics[scale=.63]{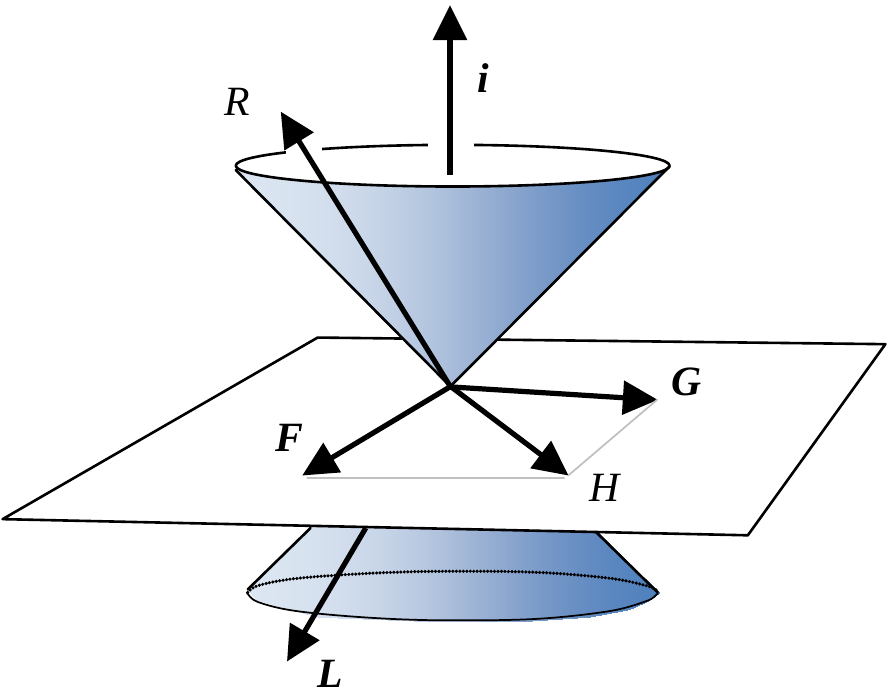}
\caption{Minkowski space and null vectors (not to scale)}
\label{fig:Minkowski2}
\end{figure} 
\noindent
{\bf Geometric interpretation within $\mathbb K$:}  
Elementary Krawtchouk matrix $\ontop 1K = H$ 
understood as a group element acting on the $\Pu\mathbb  \cong \mathbb K$
via adjoint action 
is the ``space-like'' reflection in the plane $(H,\mathbi i$ followed followed by ``time'' 
reflection in the plane perpendicular to $i$.
This follows from these three facts:
$$
HFH^{-1} = G  \qquad HGH^{-1} = F  \qquad  HiH^{-1} = -i  
$$
We have already explored the first two. The third follows easily:
$$
HiH^{-1} = {\scriptstyle\frac{1}{2}}(F+G) \; i \; (F+G)  = -{\scriptstyle\frac{1}{2}}(F+G)^2 \; i =  -i
$$
(Note that $H^{-1} = \frac{1}{2}H$.)
As a composition of two reflections, $H$ is an element of $SO(1,2;\,\mathbb R)$,
but cannot be described as a single rotation due to the topology of this group.
To complete the picture, here are the actions of $H$ on the isotropic basis elements:
$$
HLH^{-1} = -R, \qquad
HRH^{-1} = -L \qquad 
HNH^{-1} = -N
$$
(Note that they can be read off directly from figure \ref{fig:Minkowski2}.)

The above   interpretation should not be confused with the action on the spinor spaces.
Here, the elementary Krawtchouk matrix, $K^{(1)} = H$, acts on spinor space  $V=\mathbb R^2$ as a reflection 
through a line perpendicular to $[1,\,  1\!+\!\sqrt{2}]^T$, followed by scaling by $\sqrt{2}$.
The higher order Krawtchouk matrix $K^{(n)}$, as a higher powers of $H$, carries this property   
to the tensor product $V^{\odot n}$, 
where it becomes a reflection accompanied by scaling by $2^{n/2}$ 
(in agreement with $K^2 \cong I$).

\subsection{Summary} 

Let us revisit the initial examples of Section \ref{sec:examples}. 
In each we have 
the {\it distance function} $d$ (same as the total spin in Example 2 or position of the walker in Example 3),
an equivalence relation between the states:  $a\sim b$ if  $d(a)=d(b)$,
and the {\it clusters} as the equivalence classes $\mathbb Z_2^n/\sim$,
with the map $\sim \mathbb Z_2^n \to \mathbb N_n=\{0,1,...,n\}$.

\begin{figure}[h]
\centering
\includegraphics[scale=.74]{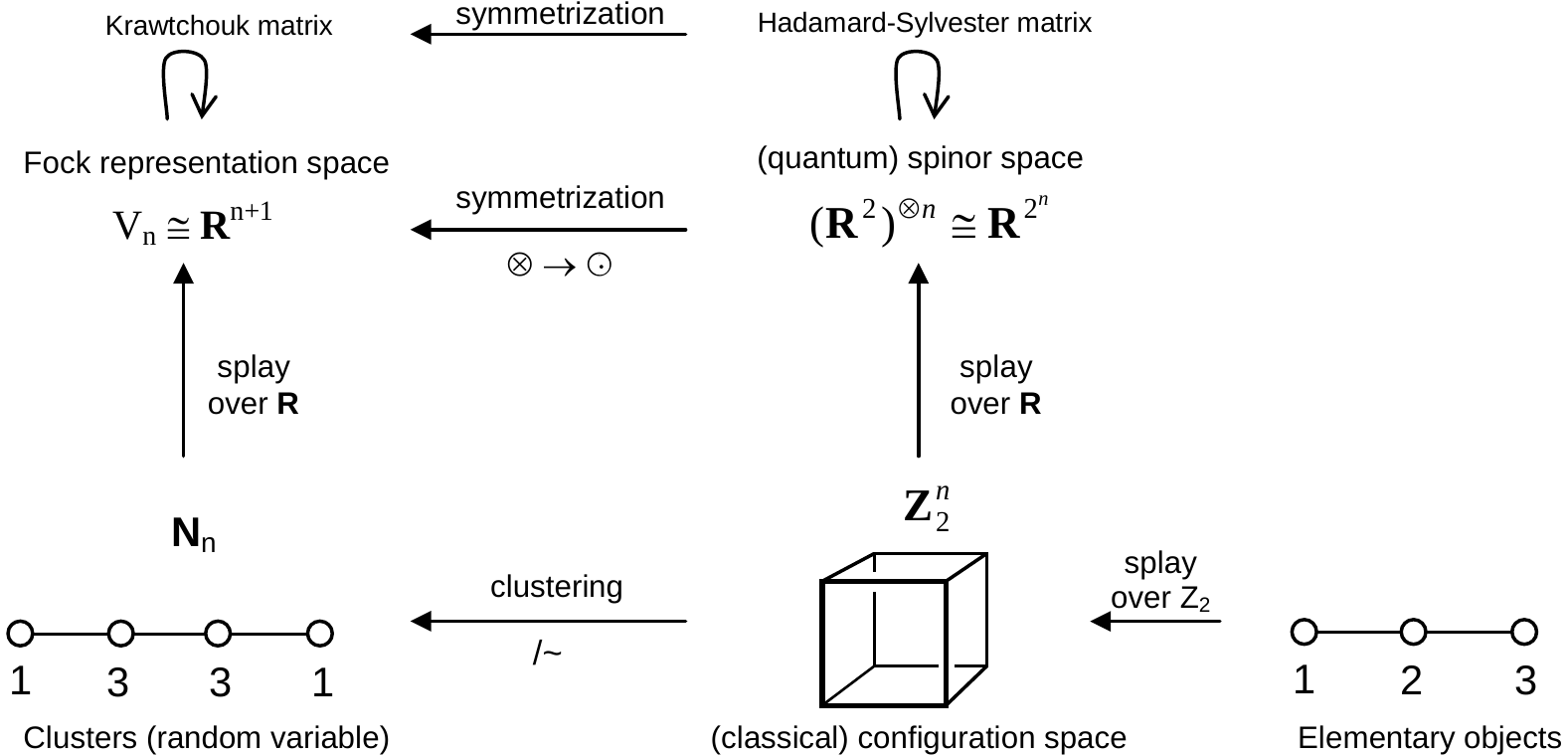}
\caption{Various objects and their relations}
\label{fig:conclusion}
\end{figure} 

These elementary concepts permeate all algebraic constructions discussed, as displayed in      
Figure \ref{fig:conclusion}. 
A {\bf controlled}  process on the state space $Z_2^n$ is computing, which becomes ``tallying'' in $\mathbb  N_n$.
A {\bf random} process on the state space $Z_2^n$ is Bernoulli walk, which becomes random variable with values in $\mathbb  Z_2^n$. 
Under tensor powers, we obtain a quantum versions of these processes: 
$\mathbb R$-quantum computing and quantum random walk,
formalized in terms of the spinor space $\mathbb F^{2^n}$.
Under symmetrization, interestingly, we merge the thing with the standard Fock representation of the algebra $sl(2)$

The $\mathbb C$-quantum computing  is obtained by replacing the real numbers $\mathbb R$ with the complex numbers
in the above constructions.
Krawtchouk matrices, as the symmetrized versions of the Hadamard-Sylvester matrices, 
may become useful in a quantum algorithm seeking some collective, order independent, evaluations.

\newpage

\appendix
\section{Krawtchouk matrices}
\label{sec:AKM}

\hrule
\bigskip
{\small
$$
\begin{aligned}
\ontop 0K &=\left[\begin{array}{r} 1 \end{array}\right]   
\\[7pt]
\ontop 1K&=\left[\begin{array}{rr} 1 &  1 \cr
                           1 & -1 \cr\end{array}\right]    
\\[7pt]
\ontop 2K&= \left[\begin{array}{rrr}  1 &  1 &  1 \cr
                             2 &  0 & -2 \cr
                             1 & -1 &  1 \cr \end{array}\right]
\\[7pt]
\ontop 3K&=
       \left[\begin{array}{rrrr}
                      1 &  1 &  1  &  1 \cr
                      3 &  1 & -1  & -3 \cr
                      3 & -1 & -1  &  3 \cr
                      1 & -1 &  1  & -1 \cr 
             \end{array}\right]
\\[7pt]
\ontop 4K&=
       \left[\begin{array}{rrrrr} 1 &  1 &  1  &  1  &  1 \cr
                      4 &  2 &  0  & -2  & -4 \cr
                      6 &  0 & -2  &  0  &  6 \cr
                      4 & -2 &  0  &  2  & -4 \cr
                      1 & -1 &  1  & -1  &  1 \cr \end{array}\right]
\end{aligned}
\qquad
\begin{aligned}
\ontop 5K&=
       \left[\begin{array}{rrrrrr} 1  &  1 &  1  &  1  &  1 &   1 \cr
                      5  &  3 &  1  & -1  & -3 &  -5 \cr
                     10  &  2 & -2  & -2  &  2 &  10 \cr
                     10  & -2 & -2  &  2  &  2 & -10 \cr
                      5  & -3 &  1  &  1  & -3 &   5 \cr
                      1  & -1 &  1  & -1  &  1 &  -1 \cr \end{array}\right]
\\[7pt]
\ontop 6K&=
       \left[\begin{array}{rrrrrrr} 
                      1  &  1 &  1  &  1  &  1 &  1  &   1 \cr
                      6  &  4 &  2  &  0  & -2 & -4  &  -6 \cr
                     15  &  5 & -1  & -3  & -1 &  5  &  15 \cr
                     20  &  0 & -4  &  0  &  4 &  0  & -20 \cr
                     15  & -5 & -1  &  3  & -1 & -5  &  15 \cr
                      6  & -4 &  2  &  0  & -2 &  4  &  -6 \cr
                      1  & -1 &  1  & -1  &  1 & -1  &   1 \cr
\end{array}\right]
\\[7pt]
\ontop 6K&=
       \left[\begin{array}{rrrrrrrr} 
                       1  &  1 &  1  &  1  &  1 &  1  &   1 & 1\cr
                       7  &  5 &  3  &  1  & -1 & -3  &  -5& -7\cr
                     21  &  9 &  1  & -3  & -3 &  1  &   9&  21\cr
                     35  &  5 & -5  & -3  &  3 &  5  &  -5& -35\cr
                     35  & -5 & -5  &  3  &  3 & -5  &  -5&  35\cr
                     21  & -9 &  1  &  3  & -3 & -1  &   9& -21\cr
                       7  & -5 &  3  & -1  & -1 &  3  &  -5&   7\cr
                       1  & -1 &  1  & -1  &  1 & -1  &   1&  -1 \cr
\end{array}\right]
\end{aligned}
$$
}
\bigskip
\hrule
\bigskip
\centerline{{\bf Table 1:} Krawtchouk  matrices}

~~

Define {\it Krawtchouk vectors} ({\it covectors})  as the columns (rows) of a Krawtchouk matrix. 
Here i an example for $n=3$
\\[3pt]

\hrule
~~\\[-12pt]
$$
\begin{matrix} \mathbf k_0 & \mathbf  k_1 &  \mathbf  k_2  &\mathbf  k_3  \\
                 \downarrow   &\downarrow     &  \downarrow     &\downarrow   \end{matrix}
\qquad\qquad\qquad\qquad\qquad\qquad\qquad\qquad\qquad
$$
\vspace{-7pt}
$$
       \left[\begin{array}{r|r|r|r}
                      1 &  1 &  1  &  1  \cr
                      3 &  1 & -1  & -3  \cr
                      3 & -1 & -1  &  3  \cr
                      1 & -1 &  1  & -1   \end{array}\right]
\qquad\qquad
       \left[ \begin{array}{rrrr} 1 &  1 &  1  &  1  \\  \hline
                      3 &  1 & -1  & -3  \\ \hline
                      3 & -1 & -1  &  3  \\ \hline
                      1 & -1 &  1  & -1   \end{array} \right]
\qquad
             \begin{matrix}  \longleftarrow & \gamma_0  \cr
                      \longleftarrow  & \gamma_1  \cr
                      \longleftarrow  & \gamma_2  \cr
                      \longleftarrow  & \gamma_3  \cr\end{matrix} 
$$

\hrule
~~

\noindent
For a fixed $n$,  Krawtchouk covectors form a dual basis with respect to basis
defined by Krawtchouk  vectors (up to a scalar):
$$
   \langle {\gamma^i},  {\mathbf  k_j} \rangle = 2^n \delta_{ij}
$$
More interestingly, 
Krawtchouk vectors, as well as Krawtchouk covectors, are mutually orthogonal in $\mathbb R^{n+1}$.
More precisely
\begin{equation}
\label{eq:ortho}
    \langle \mathbf  k_i, \; \mathbf  k_j \rangle  =  \delta_{ij} \cdot 2^n / \begin{psmallmatrix}n\\i\end{psmallmatrix} 
\qquad \hbox{and}\qquad
    \langle  \gamma^i, \; \gamma^j \rangle =  \delta_{ij} \cdot 2^n  \begin{psmallmatrix}n\\i\end{psmallmatrix} 
\end{equation}
with respect to ``binomial" Euclidean structure, 
defined for two vectors $\mathbf  a, \mathbf  b \in \mathbb R^{n+1}$
by
$$
\langle {\mathbf a},\, {\mathbf b} \rangle \ = \ \sum_i {n \choose i}^{-1} \, a_i \,  b_i  \ = \  \mathbf a^T \, \Gamma^{-1} \, \mathbf b
$$
and for covectors $\alpha=[\alpha_0,\ldots,\alpha_n]$ and
$\beta =[\beta_0,\ldots,\beta_n]^T$
as
$$
\langle \alpha,\, \beta \rangle \ = \ \sum_i\, {n \choose i} \, \alpha_i \,  \beta_i  \ = \ \alpha \, \Gamma \, \beta^T
$$
Matrices $\Gamma$ and $\Gamma^{-1}$ are   
diagonal with binomial (inverses of binomial) coefficients along the diagonals, respectively.
For $n=4$:
{\small
$$
\ontop 4\Gamma^{-1} =
      \begin{bmatrix}  1 &     &     &     &   \cr
                       & 1/4 &     &     &   \cr
                       &     & 1/6 &     &   \cr
                       &     &     & 1/4 &   \cr
                       &     &     &     & 1 \cr \end{bmatrix} 
\qquad
\ontop 4\Gamma=
      \begin{bmatrix}  1 &   &   &   &   \cr
                       & 4 &   &   &   \cr
                       &   & 6 &   &   \cr
                       &   &   & 4 &   \cr
                       &   &   &   & 1 \cr \end{bmatrix} 
$$
}
The orthogonality relations (\ref{eq:ortho}) may be better expressed
\begin{equation}
    \langle \mathbf  k_i, \; \mathbf  k_j \rangle  =  \delta_{ij} \cdot 2^n \Gamma^{-1}_{ii} 
\qquad \hbox{and}\qquad
    \langle  \gamma^i, \; \gamma^j \rangle =  \delta_{ij} \cdot 2^n  \Gamma_{ii}
\end{equation}

These properties are consequence of a matrix identity satisfied by Krawtchouk matrices:
\begin{equation}
\label{eq:matrixortho}
K^T = \Gamma^{-1} K \Gamma
\end{equation}
Indeed, by a simple algebraic manipulations and using the fact that $K^2=2^nI$, we can get
$$
K^T \Gamma^{-1} K = 2^n \Gamma^{-1} 
\qquad\hbox{and}\qquad
K \Gamma K^T = 2^n \Gamma\,,
$$ 
which are matrix versions of (\ref{eq:ortho}).
Thus Equation (\ref{eq:matrixortho}) can be called the {\it orthogonality condition}
for Krawtchouk matrices.
Matrices $S= K\Gamma$ are {it symmetric} Krawtchouk matrices,
see the next Appendix. 
\\

~~

\noindent
{\bf Exercise:}  Here is yet another interesting property, presented as a simple problem. 
Represent exponential function $f_2(i)=2^i$ 
by a covector of its values, $\alpha_i=2^{n-i}$  
Check that acting on it from the right by $K^{(n)}$ produces exponential covector $f_3$. 
For instance, 
$$
[8, 4, 2, 1] \qquad \longrightarrow \qquad [27, 9, 3, 1]
$$
Acting on a covector representing $f_3(i) =3^i$ recovers $2^i$, rescaled:
$$
[27, 9, 3, 1] \qquad \longrightarrow \qquad \sim [8, 4, 2, 1]
$$
Explain the phenomenon. 
What about other exponential functions? 
Show that Krawtchouk transformation of covectors caries exponential functions to exponential functions.
Compare with Theorem \ref{thm:binomialtransf}.

\newpage
\section{Krawtchouk matrices from Hadamard matrices}
~~
\vspace{-9pt}

\hrule

{\small
$$
\begin{aligned}
\ontop 0 S&=\left[\begin{array}{r} 1 \end{array}\right]
\\[7pt]
\ontop 1 S&=\left[\begin{array}{rr} 1 &  1 \cr
                        1 & -1 \cr\end{array}\right]
\\[7pt]
\ontop 2 S&= \left[\begin{array}{rrr}  1 &  2 &  1 \cr
                          2 &  0 & -2 \cr
                          1 & -2 &  1 \cr \end{array}\right]
\\[7pt]
\ontop 3 S&=
      \left[\begin{array}{rrrr} 1 &  3 &  3  &  1 \cr
                     3 &  3 & -3  & -3 \cr
                     3 & -3 & -3  &  3 \cr
                     1 & -3 &  3  & -1 \cr \end{array}\right]
\end{aligned}
\qquad
\begin{aligned}
\ontop 4 S&=
      \left[\begin{array}{rrrrr} 1 &  4 &  6  &  4  &  1 \cr
                     4 &  8 &  0  & -8  & -4 \cr
                     6 &  0 &-12  &  0  &  6 \cr
                     4 & -8 &  0  &  8  & -4 \cr
                     1 & -4 &  6  & -4  &  1 \cr \end{array}\right]
\\
\\
\ontop 5 S&=
      \left[\begin{array}{rrrrrr} 1  &  5 & 10 &  10  &   5 &   1 \cr
                     5  & 15 & 10 & -10  & -15 &  -5 \cr
                    10  & 10 &-20 & -20  &  10 &  10 \cr
                    10  &-10 &-20 &  20  &  10 & -10 \cr
                     5  &-15 & 10 &  10  & -15 &   5 \cr
                     1  & -5 & 10 & -10  &   5 &  -1 \cr \end{array}\right]
\end{aligned}
$$
}

\hrule
\bigskip
\centerline{{\bf Table 1:} Symmetric Krawtchouk  matrices}

~\\[-21pt]

Hadamard-Sylvester matrices are obtained by tensor powers of the fundamental Hadamard matrix $H$ (here $H_1$).
Below, we show the first three powers $H_n=H^{\otimes n}$ represented via Kronecker product 
(Kronecker product of two matrices $A$ and $B$ is a matrix obtained by multiplying every entry of $A$ by $B$).
\\[-11pt]
\def\b{\circ}   \def\a{\bullet}
$$
H_{1}=\kbordermatrix{%
  & 0 & 1  \cr
0 &\a &\a  \cr
1 & \a &\b  \cr
}
\quad
H_{2}=\kbordermatrix{%
    &0 &1 &1  &2   \cr
0   &\a &\a &\a &\a  \cr
1   &\a &\b &\a &\b   \cr 
1   &\a &\a &\b &\b   \cr
2   &\a &\b &\b &\a   \cr
}
\quad
H_{3}=\kbordermatrix{%
    &0 &1 &1  &2 &1 &2 &2 &3   \cr
0 &\a &\a &\a &\a  &\a &\a &\a &\a    \cr
1   &\a &\b &\a &\b  &\a &\b &\a &\b    \cr
1   &\a &\a &\b &\b  &\a &\a &\b &\b    \cr
2  &\a &\b &\b &\a  &\a &\b &\b &\a    \cr
1  &\a &\a &\a &\a  &\b &\b &\b &\b    \cr 
2  &\a &\b &\a &\b  &\b &\a &\b &\a    \cr 
2  &\a &\a &\b &\b  &\b &\b &\a &\a    \cr 
3  &\a &\b &\b &\a  &\b &\a &\a &\b    \cr 
} \,,
$$ 
For clarity, we use $\bullet$ for $1$ and $\circ$ for $-1$.
The columns and rows are labeled by strings $w(i)$, $i=0,...,2^n$
defined recursively by setting $w(0)=0$ and $w(2^n + k) = w(k)+1$.
(Next string is obtained by appending by it with its copy with the values increased by $1$.)
\begin{equation}
0
\ \rightarrow\
01
\ \rightarrow\
0112
\ \rightarrow\
01121223
\ \rightarrow\ 
\hbox{etc.} 
\end{equation}
Symmetric Krawtchouk matrices are reductions of Hadamard matrices:
the entries are sums of all entries that have the same labels $w$:
$$
S^{(n)}_{pq} = \sum_{a\in w^{-1}(p) \atop b\in w^{-1}(q)} H^{\otimes n}_{ab}
$$
(The problem is that Kronecker products disperse
the indices of columns and rows that would have to be 
summed up when one symmetrizes the tensor product. 
The label function $w(i)$ identifies the appropriate sets of indices.)
\\[7pt]
{\bf Remark:} The index strings define an integer sequence $w : {\mathbb N} \to {\mathbb N}: k\mapsto w(k)$
as the ``binary weight" of the integer $k$, i.e., the number of ``1'' in its 
binary expansion.  If  $k=\sum_i d_k 2^i$ then $w(k)=\sum_i d_i$.

\newpage

\section{Pascal-Krawtchouk pyramid}
 
Stacking the Krawtchouk matrices one upon another creates a pyramid of integers,
the {\bf Krawtchouk pyramid}.  It may me viewed as a 3-dimensional generalization of Pascal triangle.
In particular, its West wall coincides with the Pascal triangle.
This formations makes easier to visualize various identities held by Krawtchouk matrices.

\begin{figure}[h]
\centering
\includegraphics[scale=.5]{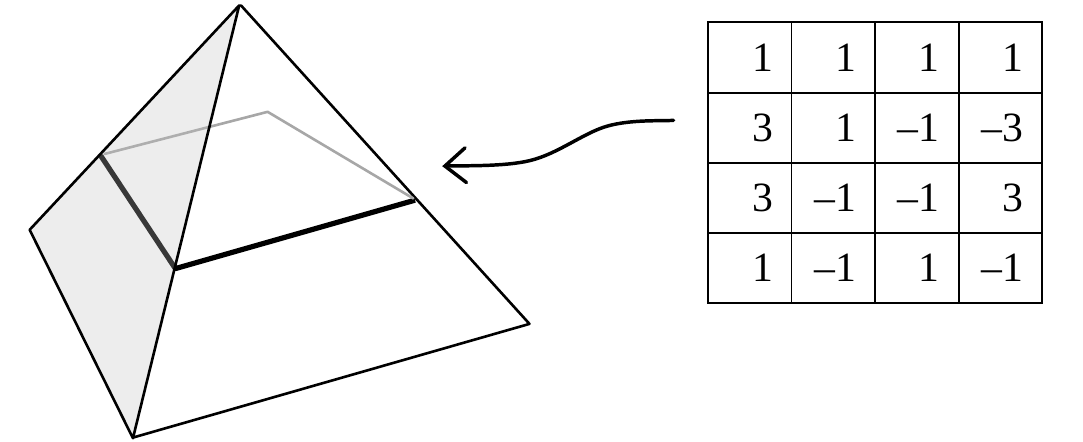}
\caption{Krawtchouk pyramid;  Krawtchouk matrices are in the horizontal planes.}
\label{gig:pyramid}
\end{figure}

\noindent 
Krawtchouk matrices satisfy the {\it cross identities} and the {\it square identity}:
\begin{equation}
\label{eq:cross}
\begin{array}{crclrrcl}
(i)  & \ontop{n} K_{i+1,j}+ \ontop nK_{ij} &=& \ontop {n+1}K_{i+1,j} 
     &(iii)& \ontop{n} K_{ij}+ \ontop nK_{i,j+1} &=&  2\ontop {n-1}K_{i,j}\\[3pt]
(ii) &\ontop nK_{i+1,j}- \ontop nK_{ij}   &=& \ontop {n+1}K_{i+1,j+1}\qquad
     &(iv)& \ontop nK_{ij}- \ontop nK_{i,j+1} &=& 2 \ontop {n-1}K_{i-1,j}\\
\end{array}
\end{equation}

\vspace{-9pt}

$$
(v) \quad  \ontop nK_{ij}+ \ontop nK_{i,j+1}  + \ontop {n}K_{i+1,j+1} = \ontop {n}K_{i,j+1} 
$$
(For a proof see Section \ref{sec:ring}.)
They may be visualized as shown in Figure \ref{fig:cross}.  
The first two relate consecutive levels of Krawtchouk pyramid.
The last states that in any square of four adjacent entries in any of the Krawtchouk matrices, three add up to the fourth.
Cutting the pyramid by planes parallel to any of the sides  
results in Pascal-like triangles with the corresponding rules,  
derived from Eq. (\ref{eq:cross}).
%
\begin{figure}[h]
\centering
\includegraphics[scale=.57]{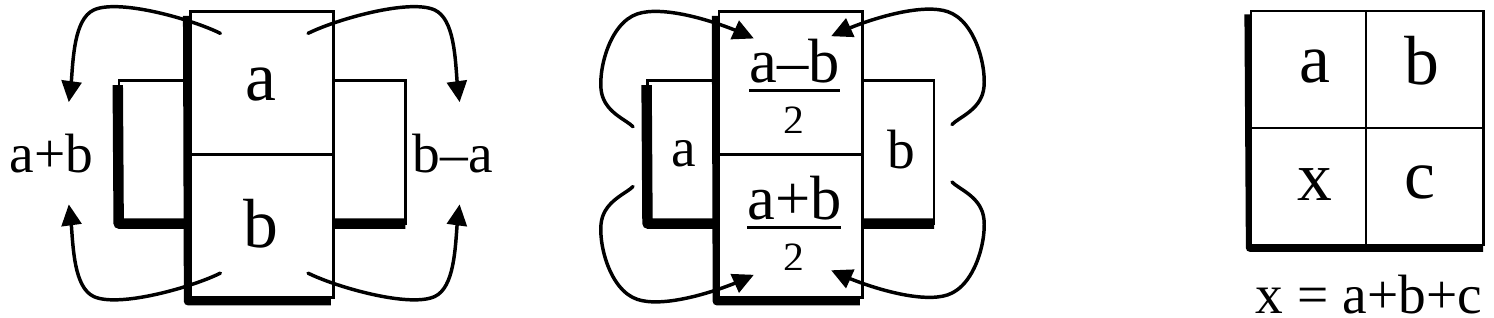}
\caption{Identities for Krawtchouk identities visualized}
\label{fig:cross}
\end{figure}

The {\bf West} wall coincides with the Pascal triangle, denoted $P^{+}_{\downarrow}(0)$.
Any plane parallel to it but cutting the pyramid at a deeper level, denoted $P^{+}_{\downarrow}(n)$,
starts with the entries of the last column on the $n$-th Krawtchouk matrix and continues down with the usual
Pascal addition rule.

The {\bf East} wall,  $P^{-}_{\downarrow}(0)$, consists of binomial coefficients with alternating signs. 
It follows the Pascal rule except of the difference replacing the sum. 
The same rule applies to parallel planes $P^{+}_{\downarrow}(n)$ immersed deeper into the pyramid.

The {\bf North} wall, $P^{+}_{\uparrow}(0)$, consists of $1$'s and its rule is an inverted Pascal rule:  
the sum of two adjacent entries equals twice the entry above them.  
The same rule holds for any parallel plane $P^{+}_{\uparrow}(n)$ starting at the $n$-th level, 

Finally, the {\bf South} wall,  $P^{-}_{\uparrow}(0)$, consists of $\{\pm 1\}$'s and its rule is again the inverted Pascal rule but the a difference 
replacing the sum. The same rule holds for any parallel plane, $P^{-}_{\uparrow}(n)$.


{\begingroup
\begin{figure}[H]
\centering
\includegraphics[scale=.75]{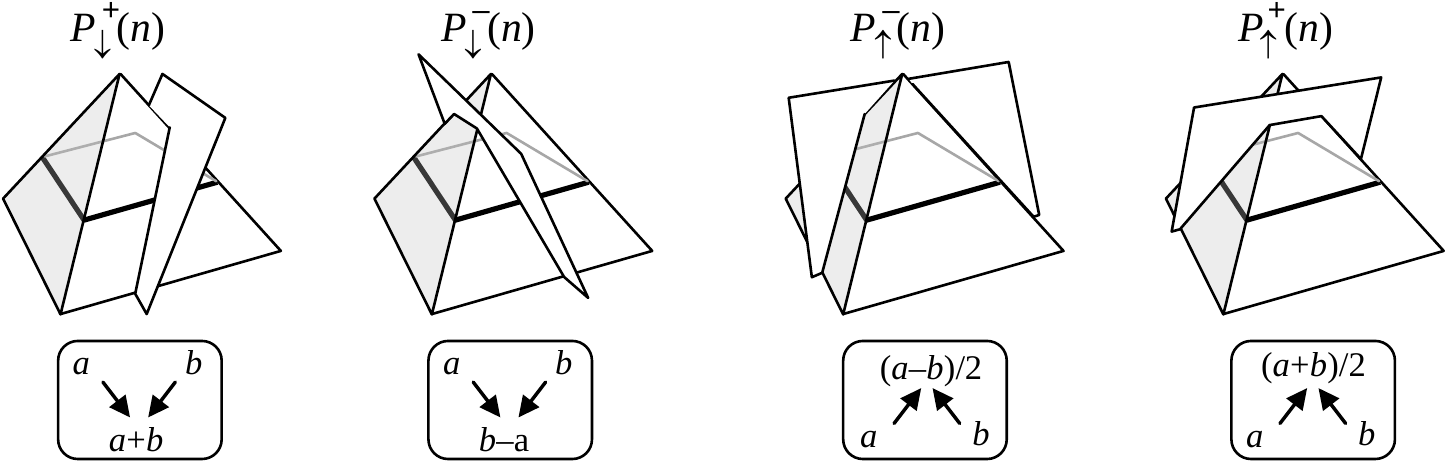}
\caption{Various Pascal-like planes through the Krawtchouk pyramid}
\label{gig:pyramids}
\end{figure}

\noindent
Numerical examples follow.
In each pair, the first (on the left) is a surface of the pyramid,
the second a plane parallel, through the pyramid.


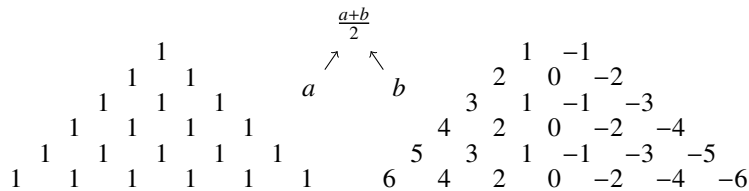
\begin{figure}[h]
\hrule
~~

{\small
\noindent
Pascal $P^+_{\downarrow}$.
Planes are oriented S-N, with the top tilted towards East:
\vspace{-12pt}
{\small 
$$
\hbox{\begingroup
\setlength{\tabcolsep}{3pt} 
\renewcommand{\arraystretch}{.8} 
\begin{tabular}{ccccccccccccccc}
 &&&&&&&1&&&&&&\\
 &&&&&&1&&1&&&&&\\
 &&&&&1&&2&&1&&&&\\
 &&&&1&&3&&3&&1&&&\\
 &&&1&&4&&6&&4&&1&&\\
 &&1&&5&&10&&10&&5&&1&\\
\end{tabular}
\endgroup}
\hspace{-1cm}
\begin{tikzpicture}[scale=1]
    \matrix (m) [ matrix of math nodes,
    row sep=.7em, column sep=.1em,  text height=1.5ex, text depth=1ex] 
{a & &  b  \\
      & a\!+\!b \\ };
\path[->]
        (m-1-1) edge (m-2-2)
        (m-1-3) edge (m-2-2);
\end{tikzpicture}   
\hspace{-1cm}
\hbox{\begingroup
\setlength{\tabcolsep}{3pt} 
\renewcommand{\arraystretch}{.8} 
\begin{tabular}{ccccccccccccccc}
 &&&&&&$1$&&$\!\!\!-1$&&&&&\\
 &&&&&$1$&&$0$&&$\!\!\!-1$&&&&\\
 &&&&$1$&&$1$&&$\!\!\!-1$&&$\!\!\!-1$&&&\\
 &&&$1$&&$2$&&$0$&&$\!\!\!-2$&&$\!\!\!-1$&&\\
 &&$1$&&$3$&&$2$&&$0$&&$\!\!\!-2$&&$\!\!\!-1$&\\
 &$1$&&$4$&&$5$&&$2$&&$\!\!\!-2$&&$\!\!\!-3$&&$\!\!\!-1$
\end{tabular}
\endgroup}
$$
}

\noindent
Pascal $P^{-}_{\downarrow}$.
Planes are oriented S-N, with the top tilted towards West:  
\vspace{-12pt}
{\small
$$
\hbox{\begingroup
\setlength{\tabcolsep}{2.5pt} 
\renewcommand{\arraystretch}{.8} 
\begin{tabular}{cccccccccccccccc}
 &&&&&&&$1$&&&&&&\\
 &&&&&&$1$&&$\!\!\!-1$&&&&&&\\
 &&&&&$1$&&$\!\!\!-2$&&$1$&&&&&\\
 &&&&$1$&&$\!\!\!-3$&&$3$&&$\!\!\!-1$&&&&\\
 &&&$1$&&$\!\!\!-4$&&$6$&&$\!\!\!-4$&&$1$&&&\\
\end{tabular}
\endgroup}
\hspace{-1cm}
{\small \begin{tikzpicture}[scale=1]
    \matrix (m) [ matrix of math nodes,
    row sep=.7em, column sep=.1em,  text height=1.5ex, text depth=1ex] 
{a & &  b  \\
      & b\!-\!a \\ };
\path[->]
        (m-1-1) edge (m-2-2)
        (m-1-3) edge (m-2-2);
\end{tikzpicture}   }
\hspace{-1cm}
\hbox{\begingroup
\setlength{\tabcolsep}{3pt} 
\renewcommand{\arraystretch}{.8} 
\begin{tabular}{ccccccccccccccc}
 &&&&&$1$&&$2$&&$1$&&&&\\
 &&&&$1$&&$1$&&$\!\!\!-1$&&$\!\!\!-1$&&&\\
 &&&$1$&&$0$&&$\!\!\!-2$&&$0$&&$1$&&\\
 &&$1$&&$\!\!\!-1$&&$\!\!\!-2$&&$2$&&$1$&&$\!\!\!-1$&\\
 &$1$&&$\!\!\!-2$&&$\!\!\!-1$&&$4$&&$\!\!\!-1$&&$\!\!\!-2$&&$1$
\end{tabular}
\endgroup}
$$
}

\noindent
Pascal $P^{-}_{\uparrow}$.
Planes are oriented E-W, with the top tilted towards North:  
\vspace{-12pt}
{\small
$$
\hbox{\begingroup
\setlength{\tabcolsep}{3pt} 
\renewcommand{\arraystretch}{.8} 
\begin{tabular}{cccccccccccccccc}
 &&&&&&&$1$&&&&&&\\
 &&&&&&$1$&&$\!\!\!-1$&&&&&&\\
 &&&&&$1$&&$\!\!\!-1$&&$1$&&&&&\\
 &&&&$1$&&$\!\!\!-1$&&$1$&&$\!\!\!-1$&&&&\\
 &&&$1$&&$\!\!\!-1$&&$1$&&$\!\!\!-1$&&$1$&&&\\
 &&$1$&&$\!\!\!-1$&&$1$&&$\!\!\!-1$&&$1$&&$\!\!\!-1$&&\\
\end{tabular}
\endgroup}
\hspace{-1cm}
\begin{tikzpicture}[scale=1]
    \matrix (m) [ matrix of math nodes,
    row sep=.7em, column sep=.1em,  text height=1.5ex, text depth=1ex] 
{ &\frac{a-b}{2} &   \\
     a &&b \\ };
\path[->]
        (m-2-1) edge (m-1-2)
        (m-2-3) edge (m-1-2);
\end{tikzpicture}   
\hspace{-1cm}
\hbox{\begingroup
\setlength{\tabcolsep}{2.4pt} 
\renewcommand{\arraystretch}{.8} 
\begin{tabular}{cccccccccccccccc}
 &&&&&$1$&&$1$&&$1$&&&&\\
 &&&&$3$&&$1$&&$\!\!\!-1$&&$\!\!\!-3$&&&\\
 &&&$6$&&$0$&&$\!\!\!-2$&&$0$&&$6$&&\\
 &&$\!10\!$&&$\!\!\!-2$&&$\!\!\!-2$&&$2$&&$2$&&$\!\!\!\!-10\!$&\\
 &$\!15\!$&&$\!\!\!-5$&&$\!\!\!-1$&&$3$&&$\!\!\!-1$&&$\!\!\!-5$&&$\!15\!$\\
$\!21\!$&&$\!\!\!-9$&&$1$&&$3$&&$\!\!\!-3$&&$\!\!\!-1$&&$9$&&$\!\!\!\!-21\!$
\end{tabular}
\endgroup}
$$
}

\noindent
Pascal $P^{+}_{\uparrow}$.
Planes are oriented E-W, with the top tilted towards South:  
\vspace{-12pt}
{\small 
$$
\hbox{\begingroup
\setlength{\tabcolsep}{3pt} 
\renewcommand{\arraystretch}{.8} 
\begin{tabular}{cccccccccccccccc}
 &&&&&&&$1$&&&&&&\\
 &&&&&&$1$&&$1$&&&&&&\\
 &&&&&$1$&&$1$&&$1$&&&&&\\
 &&&&$1$&&$1$&&$1$&&$1$&&&&\\
 &&&$1$&&$1$&&$1$&&$1$&&$1$&&&\\
 &&$1$&&$1$&&$1$&&$1$&&$1$&&$1$&&\\
\end{tabular}
\endgroup}
\hspace{-1cm}
\begin{tikzpicture}[scale=1]
    \matrix (m) [ matrix of math nodes,
    row sep=.7em, column sep=.1em,  text height=1.5ex, text depth=1ex] 
{ &\frac{a+b}{2} &   \\
     a &&b \\ };
\path[->]
        (m-2-1) edge (m-1-2)
        (m-2-3) edge (m-1-2);
\end{tikzpicture}   
\hspace{-1cm}
\hbox{\begingroup
\setlength{\tabcolsep}{2.7pt} 
\renewcommand{\arraystretch}{.8} 
\begin{tabular}{ccccccccccccccc}
 &&&&&&$1$&&$\!\!\!-1$&&&&&\\
 &&&&&$2$&&$0$&&$\!\!\!-2$&&&&\\
 &&&&$3$&&$1$&&$\!\!\!-1$&&$\!\!\!-3$&&&\\
 &&&$4$&&$2$&&$0$&&$\!\!\!-2$&&$\!\!\!-4$&&\\
 &&$5$&&$3$&&$1$&&$\!\!\!-1$&&$\!\!\!-3$&&$\!\!\!-5$&\\
 &$6$&&$4$&&$2$&&$0$&&$\!\!\!-2$&&$\!\!\!-4$&&$\!\!\!-6$\\
\end{tabular}
\endgroup}
$$
}

}
\hrule
~~
\caption{Various Pascal-like triangles appearing in the Krawtchouk pyramid }
\end{figure}
\endgroup}

\newpage

\section{Subspaces of binary spaces}

Let $V = \mathbb Z_2^n$ be the $n$-dimensional space over the Galois field $\mathbb Z_2 \equiv \{0,1\}$ 
with the standard basis $\{e_1,...,e_n\}$ and the standard inner product $\langle\,\cdot\, , \,\cdot\, \rangle$ understood modulo 2.
Clearly, every vector is of the form
$$
      v = v^i e_i  
$$
with $v^i\in \mathbb Z_2$.
Define the weight convector as $\omega = \varepsilon_1+ \varepsilon_2 +...+   \varepsilon_n$,     
where $\{\varepsilon_ i\}$ is the dual basis.   In the matrix notation
$$
                                 \omega = [1,1,...,1]
$$
The {\bf weight} of a vector is defined as a map
$$
       v \mapsto   \langle \omega, v\rangle  = \hbox{number of ones in\ } v
$$
The space $V = \mathbb Z_2^n$ may be viewed as an $n$-dimensional cube 
and the weight of a vector as the graph-theoretic distance of the vertex $v$ from $0$.
Define the {\bf weight character} of a subspace $W < V$ as a vector $\mathbf W \in \mathbb Z^n$,  
the $i$-th component of which is defined as  the number of vectors in $W$ of weight $i$:
$$
    \mathbf W_i 
  = \hbox{card\,} \{v\in W\; | \; \langle \omega, v\rangle = i \} 
$$
Clearly, $\mathbf W_0=1$ for any subspace $W$.  
MacWilliams' theorem-- originally expressed in the context of linear codes and in a combinatorial language   \cite{MWS} -- 
may be formulated  in a purely geometrical language  for codes over $\mathbb Z_2$ as follows:

~\\ \noindent
{\bf Theorem.}
{\it 
The weight character of an orthogonal complement of the subspace $W$ of a binary space $V$ is a Krawtchouk transform 
of the weight character of $W$, rescaled by the dimension of $W^\bot$:
$$
(\hbox{\rm dim} \,  W^\bot )\cdot  \mathbf W^\bot  = K \mathbf W
$$
where $K$ is the n-th Krawtchouk matrix, $n=\hbox{\rm dim}\, V$.}
\\

Figure \ref{fig:spaces}  illustrates a few examples for the 3-dimensional space.
For instance, the middle pair represents:
$$
\begin{aligned}
W   &= \hbox{span}\,\{e_1+e_2\}  = \{0,\, e_1+e_2\} \\
W^\bot &= \hbox{span}\, \{e_1+e_2, \,e_3\}   =  \{0, \,e_3, \,e_2+e_3, \,e_1+e_2+e_3\}
\end{aligned}
$$
Hence 
$$
\mathbf W = [1,0,1,0]^T\qquad \mathbf W^\bot = [1,1,1,1]^T
$$
($T$ denotes transpose).  Indeed:
$$
    2 \cdot \begin{bmatrix}1\\ 1\\ 1\\ 1\end{bmatrix}  
      =  \left[\begin{array}{rrrr} 1 &  1 &  1  &  1 \cr
                                              3 &  1 & -1  & -3 \cr
                                              3 & -1 & -1  &  3 \cr
                                              1 & -1 &  1  & -1 \cr \end{array}\right] \;
                     \begin{bmatrix}1\\ 0\\ 1\\ 0\end{bmatrix} 
$$

{\begingroup
\begin{figure}[H]
\centering
\includegraphics[scale=.75]{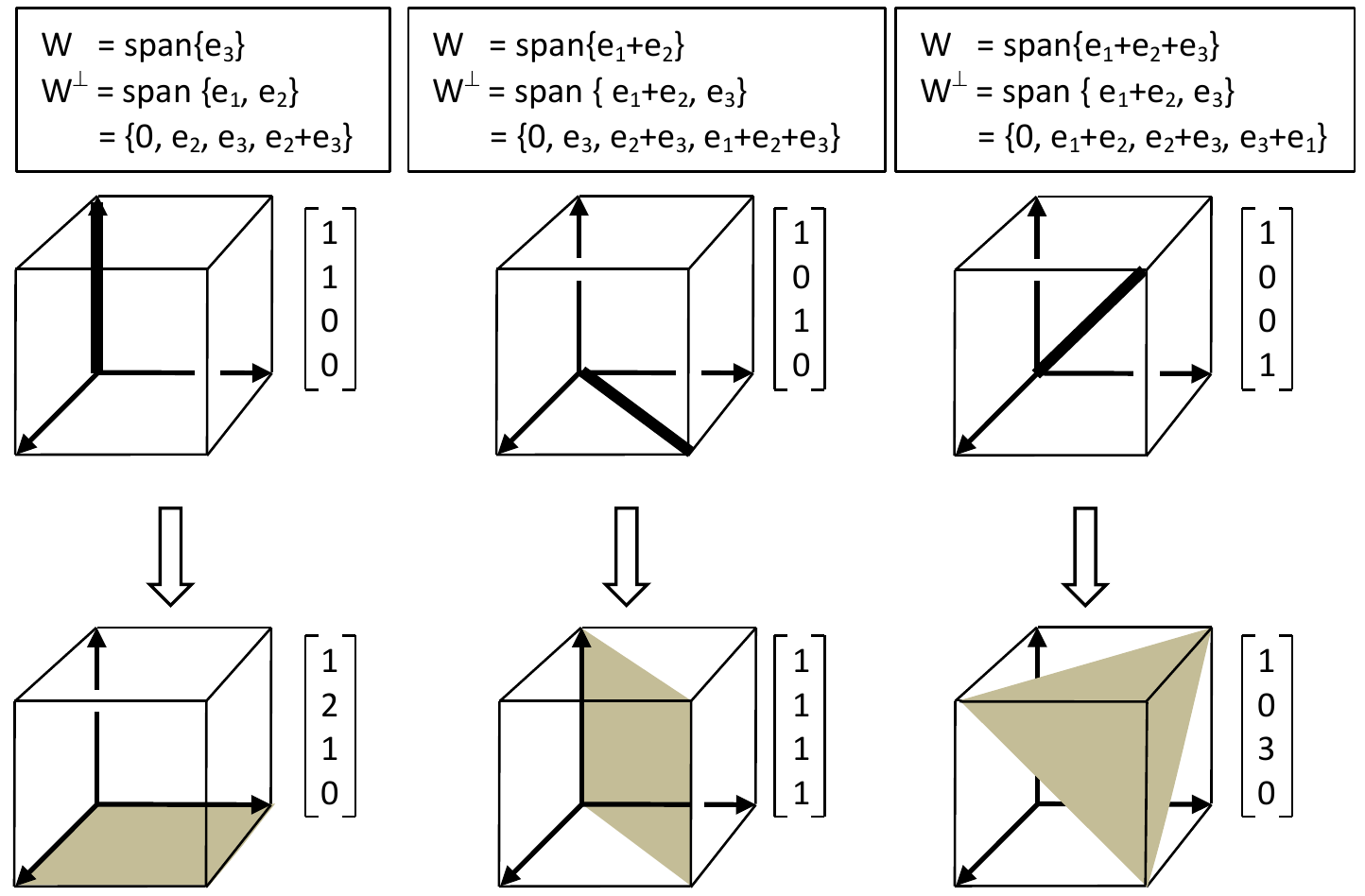}
\caption{Orthogonal subspaces and the characters}
\label{fig:spaces}
\end{figure}
Note that the subspaces span by a selection of the basis vectors have their characteristic weight vector a binomial distribution.
Quite interestingly, for such spaces the ``skew-diagonalization'' of the Krawtchouk matrix is a special case of MacWilliams' theorem.


\end{document}